\documentclass{preprint}
\usepackage{tikz}
\usepackage{wrapfig}
\usepackage{colortbl}
\usepackage{booktabs}
  \usepackage{times}%
  \usepackage{amssymb}%
  \usepackage{amsmath}%
  \usepackage{mhequ}
  \usepackage{mhenvs}
  \usepackage{Symbols}
  \usepackage{booktabs}

\newtheorem{assumption}{Assumption}

\def\E{\mathbf{E}}
\def\TV{\mathrm{TV}}
\def\Hf{H_{\!f}}
\def\hf{{\textstyle{1\over 2}}}
\def\Span{\mathrm{span}}
\def\HO{\mathbf{H}_0}
\def\HH{\mathbf{H}_1}
\def\HU{\mathcal{H}_0}
\def\cut{\eta}
\def\ps{\theta}
\def\VV{V_{\!\mathit{eff}}}

\usetikzlibrary{snakes}

\begin{document} 

  \title{How hot can a heat bath get?}
  \author{Martin~Hairer}
  \institute{Mathematics Institute, The University of Warwick, Coventry CV4 7AL, UK
   \\ \email{M.Hairer@Warwick.ac.uk}}
  \titleindent=0.65cm

  \maketitle
  \thispagestyle{empty}

\begin{abstract}\parindent1em
We study a model of two interacting Hamiltonian particles subject to a common potential
in contact with two Langevin heat reservoirs: one at finite and one
at infinite temperature. This is a toy model for `extreme' non-equilibrium statistical mechanics.
We provide a full picture of the long-time behaviour of such a system, including
the existence / non-existence of a non-equilibrium steady state, the precise tail behaviour of the energy in 
such a state, as well as the speed of convergence toward the steady state.

Despite its apparent simplicity, this model exhibits
a surprisingly rich variety of long time behaviours, depending on the parameter regime: if the surrounding potential is `too stiff', then no stationary
state can exist. In the softer regimes, the tails of the energy in the stationary state can be either algebraic, 
fractional exponential, or exponential. Correspondingly, the speed of convergence to the stationary state can be either algebraic,
stretched exponential, or exponential. Regarding both types of claims, we obtain matching upper and lower bounds.\\[1mm]
\end{abstract}

\tableofcontents

\section{Introduction}

The aim of this work is to provide a detailed investigation of the dynamic and the long-time behaviour of the following model.
Consider two point particles moving in a potential $V_1$ and interacting through a harmonic force, that is the Hamiltonian
system with Hamiltonian
\begin{equ}[e:Hamiltwo]
H(p,q) = {p_0^2 + p_1^2 \over 2} + V_1(q_0) + V_1(q_1) + V_2(q_0 - q_1)\;,\quad V_2(q) = {\alpha\over 2}q^2\;.
\end{equ}
We assume that the first particle is in contact with a Langevin heat path at temperature $T > 0$. The second particle is also assumed
to have a stochastic force acting on it, but no corresponding friction term, so that it is at `infinite temperature'.
The corresponding equations of motion are
\begin{equs}[2]
dq_i &= p_i\,dt \;,\qquad i=\{0,1\}\;,\label{e:model}\\
dp_0 &= -V_1'(q_0)\,dt + \alpha(q_1 - q_0)\,dt - \gamma\,p_0\,dt &&+ \sqrt{2\gamma T}\,dw_0(t) \;, \\
dp_1 &= -V_1'(q_1)\,dt + \alpha(q_0 - q_1)\,dt &&+ \sqrt{2\gamma T_\infty} \,dw_1(t)\;,
\end{equs}
where $w_0$ and $w_1$ are two independent Wiener processes. Although we use the symbol $T_\infty$ 
in the diffusion coefficient appearing in the second oscillator, this should \textit{not} be interpreted as a
physical temperature since the corresponding friction term is missing, so that detailed balance does not hold, even if $T_\infty = T$. 
We also assume without further mention throughout this article that the parameters $\alpha$, $\gamma$, 
$T$ and $T_\infty$ appearing in the model \eref{e:model} are all strictly positive.

The equations of motion \eref{e:model} determine a diffusion on $\R^4$ with generator $\CL$ given by
\begin{equ}[e:defL]
\CL = X^H - \gamma p_0\d_{p_0} + \gamma \bigl(T_0 \d_{p_0}^2 + T_\infty \d_{p_1}^2\bigr)\;,
\end{equ}
where $X^H$ is the Liouville operator associated to  $H$, \ie the first-order differential
operator corresponding to the Hamiltonian vector field. It is easy to show that \eref{e:model} has a unique
global solution for every initial condition since the evolution of the total energy is controlled by
\begin{equ}[e:genH]
\CL H = \gamma(T_0 + T_\infty) - \gamma p_0^2\;,
\end{equ}
and so $\E H(t) \le H(0)\exp\bigl(\gamma(T_0 + T_\infty)t\bigr)$. 
Schematically, the system under consideration can thus be depicted as follows, where we show the three
terms contributing to the change of the total energy:
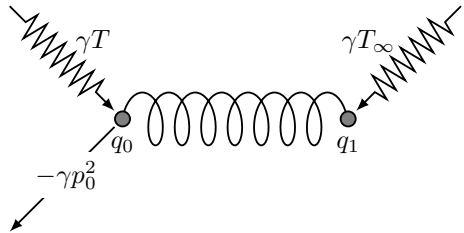
\begin{center}
\begin{tikzpicture}[segment amplitude=10pt, line width=0.7pt] 
\draw[-latex,shorten >=0.15cm, snake=zigzag,line before snake=0.3cm,line after snake=0.5cm, segment amplitude=5pt, segment length=5pt] 
	(-0.5,2.5) -- node[anchor=south west] {$\gamma T$} (1,1) ;
\draw[-latex,shorten >=0.15cm, snake=zigzag,line before snake=0.3cm,line after snake=0.5cm, segment amplitude=5pt, segment length=5pt] 
	(5.5,2.5) -- node[anchor=south east] {$\gamma T_\infty$} (4,1) ;
\draw[-latex,shorten <=0.15cm] (1,1) -- node[fill=white] {$-\gamma p_0^2$} (-0.5,-0.5);

\draw[snake=coil] (1,1) node[below=3pt] {$q_0$}  --  (4,1) node[below=3pt] {$q_1$}; 
\path[draw=black, fill=gray] (1,1) circle (1mm);
\path[draw=black, fill=gray] (4,1) circle (1mm);
\end{tikzpicture} 
\end{center}

This model is very closely related to the toy model of heat conduction previously studied by various authors in \cite{EPR99NES,EPR99EPN,EckHai00NES,ReyTho00ABO,ReyTho02ECT,EckHai03SPH,Car07:1076,HaiMat08:??} consisting
of a chain of $N$ anharmonic oscillators 
coupled at its endpoints to two heat baths at possibly different temperatures. The main difference is that the present model
does not have any friction term on the second particle. 
This is similar in spirit to the system considered in \cite{MTV02:123,DMPTV07:439}, where the authors study the
stationary state of a `resonant duo' with forcing on one degree of freedom and dissipation on another one.
Because of this lack of dissipation, even the existence of a stationary state is not
obvious at all in such a system. Indeed, if the coupling constant $\alpha$ is equal to zero, one can easily check that the invariant measure
for \eref{e:model} is given by $\exp(-(p_0^2/2 + V_1(q_0))/T)\,dp_0\,dp_1\,dq_0\,dq_1$, which is obviously not integrable. 

One of the main questions of interest for such a system is therefore to understand the mechanism of energy dissipation. 
In this sense, this is a prime example of a `hypocoercive' system where the \textit{dissipation} 
mechanism does not act on all the degrees
of freedom of the system directly, but is transmitted to them indirectly through the dynamic \cite{Vil07:257,Vil08:??}. 
This is somewhat analogous to `hypoelliptic' systems, where it is the \textit{smoothing} mechanism 
that is transmitted to all degrees of freedom
through the dynamic. The system under consideration happens to be hypoelliptic as well, but 
this is not going to cause any particular difficulty and will not be the main focus of the present work.

Furthermore, since one of the heat baths 
is at `infinite' temperature, even if a stationary state exists, one would not necessarily expect it 
to behave even roughly like $\exp(-\beta H)$ for some effective inverse temperature $\beta$. 
It is therefore of independent interest to study the tail behaviour of the energy of \eref{e:model} 
in its stationary state. 

In order to simplify our analysis, we are going to limit our investigation to one of the simplest possible
cases, where $V_1$ is a perturbation of a homogeneous potential. 
More precisely, we assume that $V_1$ is an even function of class $\CC^2$ such that
\begin{equ}
V_1(x) = {|x|^{2k}\over 2k} + R_1(x)\;,
\end{equ}
with a remainder term $R_1$ such that
\begin{equ}
\sup_{x \in \R \setminus [-1,1]} \sup_{m \le 2} {|R_1^{(m)}(x)| \over |x|^{2k-1-m}}< \infty\;.
\end{equ}
Here, $k \in \R$ is a parameter describing the `stiffness' of the individual oscillators. 
(In the case $k = 0$, we assume that $V_1(x) = C + R_1(x)$ for some constant $C$.)

\begin{wrapfigure}{r}{5cm}
\begin{tikzpicture}[segment amplitude=8pt, segment length=8pt, line width=0.7pt, scale=0.8] 
\draw[-latex,shorten >=0.15cm, snake=zigzag,line before snake=0.3cm,line after snake=0.5cm, segment amplitude=5pt, segment length=5pt] 
	(-0.5,2.5) --  node[anchor=south west] {$\gamma T_0$} (1,1) ;
\draw[-latex,shorten >=0.15cm, snake=zigzag,line before snake=0.3cm,line after snake=0.5cm, segment amplitude=5pt, segment length=5pt] 
	(5.5,2.5) --  node[anchor=south east] {$\gamma T_1$} (4,1) ;
\draw[-latex,shorten <=0.15cm] (1,1) -- node[fill=white] {$-\gamma p_0^2$} (-0.5,-0.5);
\draw[-latex,shorten <=0.15cm] (4,1) -- node[fill=white] {$-\gamma p_1^2$} (5.5,-0.5);

\draw[snake=coil] (1,1) node[below=3pt] {$q_0$}  --  (4,1) node[below=3pt] {$q_1$}; 
\path[draw=black, fill=gray] (1,1) circle (1mm);
\path[draw=black, fill=gray] (4,1) circle (1mm);
\end{tikzpicture} 
\end{wrapfigure}
In the case where both ends of the chain are at finite temperature (which would correspond to the situation depicted on the right), it was shown in 
\cite{EPR99NES,EckHai00NES,ReyTho02ECT,Car07:1076} that, provided that the coupling potential $V_2$ grows at least as fast
 at infinity as the pinning potential $V_1$ and that the latter grows at least linearly (\ie provided that ${1\over 2} \le k \le 1$ with our notations), 
the Markov semigroup associated to the model has a unique invariant measure 
$\mu_\star$ and its transition probabilities converge to $\mu_\star$ at
exponential speed. One can actually show even more, namely that the Markov semigroup 
consists of \textit{compact} operators in some suitably weighted space of functions.

Intuitively, the condition that $V_2$ grows at least as fast as $V_1$ can be understood by the fact that in this case, at high energies,
the interaction dominates so that no energy can get `trapped' in the system. 
Therefore, the system is sufficiently stiff so that if the 
energy of any one of its oscillators is large, then the energy of all 
of the oscillators must be large after a very short time. As a consequence, the system behaves like a `molecule' at some effective temperature that moves in the
global potential $V_1$.
While the arguments presented in \cite{ReyTho02ECT,Car07:1076} do not cover 
the case of one of the heat baths being at infinite temperature, it is nevertheless possible to show that in this case,
 the Markov semigroup $\CP_t$ generated by solutions to \eref{e:model} 
behave qualitatively like in the case of finite temperature. In particular, if $V_1$ grows at least linearly at infinity, the system
possesses a spectral gap in a space of functions weighted by a weight function `close to' $\exp(\beta_0 H)$ for some $\beta_0>0$.

This discussion suggest that:
\begin{claim}
\item[1.] If $V_2 \gg V_1 \gg 1$, our toy model can sustain arbitrarily large energy currents.
\item[2.] In this case, even though the heat bath to the right is at infinite temperature, the system
stabilises at some finite `effective temperature', as expressed by the fact that $H$ has finite exponential moments 
under the invariant measure.
\end{claim}

This is in stark contrast with the behaviour
encountered when $V_1$ grows faster than $V_2$ at infinity. In this case, the interaction between neighbouring particles is suppressed
at high energies, which precisely favours the trapping of energy in the bulk of the chain. It was shown in \cite{HaiMat08:??} 
that this can lead in many
cases to a loss of compactness of the semigroup generated by the dynamic and the appearance of essential spectrum at $1$.
This is a manifestation of the fact that energy transport is very weak in such systems, due to the appearance of `breathers',
localised structures that only decay very slowly \cite{MacAub94:1623}.
In this case, one expects that the long-time behaviour of
\eref{e:model} depends much more strongly on the fine details of the model. For example, regarding the finiteness of the `temperature'
of the second oscillator, one may introduce the following notions by increasing order of strength:
\begin{claim}
\item[1.] There exists an invariant probability measure $\mu_\star$ for \eref{e:model}, that is a positive solution to $\CL^* \mu_\star = 0$.
\item[2.] There exists an invariant probability measure $\mu_\star$ and the average energy of the second oscillator is finite under $\mu_\star$.
\item[3.] There exists an invariant probability measure $\mu_\star$ and the energy of the second oscillator has some finite exponential moment under $\mu_\star$.
\end{claim}
We will show that it is possible to find parameters such that the second oscillator does not have finite temperature
according to \textit{any} of these notions of finiteness. On the other hand, it is also possible to find parameters such that it 
does have finite temperature according to some notions and not to others.

It turns out that, maybe rather surprisingly for such a simple model, there are \textit{five} different critical value for the strength
$k$ of the pinning potential $V_1$ that separate between qualitatively different behaviours regarding both the integrability properties
of the invariant measure $\mu_\star$ and the speed of convergence of transition probabilities towards it. These critical values are $k=0$, $k = {1\over 2}$,
$k=1$, $k = {4\over 3}$, and $k=2$.
More precisely, there exists a constant $\hat C>0$ such that, setting
\begin{equ}[e:defzetastarkappa]
\zeta_\star = {3\over 4}{\alpha^2 \hat C - T_\infty \over T_\infty}\;,\qquad \kappa = {2\over k}-1\;,
\end{equ}
the results in this article can be summarised as follows:

\begin{theorem}\label{theo:main}
The integrability properties of the invariant measure $\mu_\star$ for \eref{e:model} and the speed of convergence of 
transition probabilities of \eref{e:model} toward $\mu_\star$ can be described by the following table:
\begin{center}
\doublerulesep8mm
\begin{tabular}{@{\ \rule[-1.5mm]{0mm}{5mm}}cccc@{\ }}
\toprule
Parameter range & Integrability of $\mu_\star$ & Convergence speed & Prefactor\\ \midrule
$k > 2$ & ---  & --- & ---\\
$k = 2$, $T_\infty > \alpha^2 \hat C$ & --- & --- & ---\\
$k = 2$, $T_\infty < \alpha^2 \hat C$ & $H^{\zeta_\star \pm \eps}$ & $t^{-\zeta_\star \pm \eps}$ & $H^{\zeta_\star + \eps + 1}$ \\
${4\over 3} \le k < 2$ & $\exp(\gamma_\pm H^\kappa)$ & $\exp\bigl(- \gamma_\pm t^{\kappa/(1-\kappa)}\bigr)$ & $\exp(\delta H^\kappa)$\\
$1 < k \le {4\over 3}$ & $\exp(\gamma_\pm H^\kappa)$ & $\exp(- \gamma_\pm t)$ & $\exp(\delta H^{1-\kappa})$ \\
\rowcolor[gray]{.9} \rule[-2mm]{0mm}{6mm}$k = 1$ & $\exp(\gamma_\pm H)$ & $\exp(- \gamma_\pm t)$ & $H^\eps$ \\ 
${1\over 2} \le k < 1$ & $\exp(\gamma_\pm H)$ & $\exp(- \gamma_\pm t)$ & $\exp(\delta H^{{1\over k}-1})$ \\
$0 < k \le {1\over 2}$ & $\exp(\gamma_\pm H)$ & $\exp\bigl(- \gamma_\pm t^{k/(1-k)}\bigr)$ & $\exp(\delta H)$ \\
$k \le 0$ & --- & --- & ---\\
\bottomrule
\end{tabular}
\end{center}
Here, the  symbol `---' means that no invariant probability measure exists for the corresponding
range of parameters. Whenever there exists a (necessarily unique) invariant measure $\mu_\star$, we indicate
integrability functions $I_\pm(H)$, convergence speeds $\psi_\pm(t)$ and a prefactor $K(H)$. The constant $\eps$ can be made
arbitrarily small, whereas the constants $\gamma_+$, $\gamma_-$ and $\delta$ are fixed and depend on the fine details of the model.
For each line in this table, the following statements hold:
\begin{claim}
\item One has $\int_{\R^4} I_+(H(x))\,\mu_\star(dx) = +\infty$, but $\int_{\R^4} I_-(H(x))\,\mu_\star(dx) < +\infty$.
\item There exists a constant $C$ such that
\begin{equ}[e:upperboundTV]
\|\CP_t(x,\cdot\,) - \mu_\star\|_\TV \le C K(H(x)) \psi_+(t)\;,
\end{equ}
for every initial condition $x \in \R^4$ and every time $t \ge 0$.
\item For every initial condition $x \in \R^4$, there exits a constant $C_x$ and a sequence of times $t_n$ 
increasing to infinity such that
\begin{equ}
\|\CP_{t_n}(x,\cdot\,) - \mu_\star\|_\TV \ge C_x \psi_-(t_n)\;,
\end{equ}
for every $n$.
\end{claim}
\end{theorem}

\begin{remark}
The case $k=2$ and $T_\infty = \alpha^2 \hat C$ is not covered by these results. We expect that the system admits
no invariant probability measure in this case.
\end{remark}

\begin{remark}
For $k \in (0,{1\over 2})$, even the gradient dynamic fails
to exhibit a spectral gap. It is therefore not surprising (see for example \cite{HelNie05HES}) that in this case we see again subexponential 
relaxation speeds.
\end{remark}

\begin{remark}
This table exhibits a symmetry $\kappa \leftrightarrow k$ and $H^\kappa \leftrightarrow H$  around $k = 1$
(indicated by a grayed row in the table). The reason for this symmetry will be explained in Section~\ref{sec:heuristics} below.
If we had chosen $V_1(x) = K\log x + R_1(x)$ in the case $k=0$, this symmetry would have extended to this case, via the
correspondence $\zeta_\star \leftrightarrow {K \over T+T_\infty} - {1\over 2}$.
\end{remark}

\begin{remark}
It follows from \eref{e:upperboundTV} that the time it takes for the transition probabilities starting from $x$ to 
satisfy $\|\CP_t(x,\cdot\,) - \mu_\star\|_\TV \le {1\over 2}$, say, is bounded by $H(x)^{2-{2\over k}}$ for $k \in (1,2)$ and
by $H(x)^{{1\over k}-1}$ for $k \in (0,1)$. These bounds are expected to be sharp in view of the heuristics given
in Section~\ref{sec:heuristics} below.
\end{remark}

\begin{remark}
Instead of considering only distances in total variation between probability measures, we could also have obtained
bounds in weighted norms, similarly to \cite{DFG06SG}.
\end{remark}

\begin{remark}
The operator \eref{e:defL} appears very closely related to the kinetic Fokker-Planck operator
\begin{equ}
\CL_{V,2} = p\,\d_q - \nabla V(q)\,\d_p - \gamma p\,\d_p + \d_p^2\;,
\end{equ}
for the potential $V(q_0,q_1) = V_1(q_0) + V_1(q_1) + {\alpha \over 2}(q_0-q_1)^2$. The fundamental difference
however is that there is a lack of friction on the second degree of freedom. The effect of this is dramatic, since the
results from \cite{HerNie02IHA} (see also \cite{DesVil01:1}) show that one has exponential return to equilibrium for the kinetic
Fokker-Planck operator in the case $k \ge 1$, which is clearly not the case here.

Finally, the techniques presented in this article also shed some light on the mechanisms at play in the Helffer-Nier 
conjecture \cite[Conjecture~1.2]{HelNie05HES}, namely that the long-time behaviour of the Fokker-Planck operator without inertia
\begin{equ}
\CL_{V,1} = -\nabla V(q)\,\d_q + \d_q^2\;,
\end{equ}
is qualitatively the same as that of the kinetic Fokker-Planck operator.
If $V$ grows faster than quadratically at infinity (so that in particular $\CL_{V,1}$ has a spectral gap), 
then the deterministic motion on the energy levels gets increasingly
fast at high energies, so that the angular variables get washed out and the heuristics from Section~\ref{sec:heuristic1}
below suggests that the total energy of the system behaves like the square of an Ornstein-Uhlenbeck 
process, thus leading to a spectral gap for $\CL_{V,2}$ as well.

If on the other hand $V$ grows slower than quadratically at infinity, then the motion of the momentum variable 
happens on a faster timescale at high energies than that of the position variable. The heuristics from Section~\ref{sec:heuristic2}
below then suggests that the dynamic corresponding to $\CL_{V,2}$ is indeed very well approximated at high energies
by that corresponding to $\CL_{V,1}$.

These considerations suggest that any counterexample to the Helffer-Nier conjecture would come from a potential that has
very irregular (oscillating) behaviour at infinity, so that none of these two arguments quite works. On the other hand, any proof
of the conjecture would have to carefully glue together both arguments.
\end{remark}

The structure of the remainder of
 this article is the following. First, in Section~\ref{sec:heuristics}, we derive in a heuristic way reduced equations
for the energies of the two oscillators. While this section is very far from rigorous, 
it allows to understand the results presented above by linking the behaviour of \eref{e:model} to that of the diffusion
\begin{equ}
dX = -\eta X^\sigma \, dt + \sqrt{2}\,dW(t)\;,\quad X\ge 1\;,
\end{equ}
for suitable constants $\eta$ and $\sigma$.

The remainder of the article is devoted to the proof of Theorem~\ref{theo:main},
which is broken into five sections. 
In Section~\ref{sec:Lyap}, we introduce the technical tools
that are used to obtain the above statements. These tools are technically quite straightforward
and are all based on the existence of test functions
with certain properties. The whole art is to construct suitable test functions in a relatively systematic manner.
This is done by refining the techniques developed in \cite{HaiMat08:??} and based on ideas from homogenisation theory.

In Section~\ref{sec:existNonExist}, we proceed to showing that $k = 2$ and $T_\infty = \alpha^2 \hat C$ is
the borderline case for the existence of an invariant measure. In Section~\ref{sec:Integral}, we then show sharp
integrability properties of the invariant measure for the regime $k > 1$
when it exists. This will imply in particular that even though the effective temperature
of the first oscillator is always finite (for whatever measure of finiteness), the one of the second oscillator need not necessarily be.
In particular, note that it follows from Theorem~\ref{theo:main} that 
the borderline case for the integrability of the energy of the second oscillator in the invariant
measure is given by $k=2$ and $T_\infty = {7\over 3}\alpha^2 \hat C$. These two sections form the `meat' of the paper.

In Section~\ref{sec:Convergence}, we make use of the integrability results obtained previously
in order to obtain bounds both from above and from below
on the convergence of transition probabilities
towards the invariant measure. The upper bounds are based on a recent criterion from \cite{DFG06SG,BakCatGui08:727},
while the lower bounds are based on a simple criterion that exploits the knowledge that certain functions of the energy 
fail to be integrable in the invariant measure. Finally, in Section~\ref{sec:smallk}, we obtain the results for the case $k \le 1$.
While these final results are based on the same techniques as the remainder of the article, the construction of the relevant test
functions in this case in inspired by the arguments presented in \cite{ReyTho02ECT,Car07:1076}.

\subsection{Notations}

In the remainder of this article, we will use the symbol $C$ to denote a generic strictly positive
constant that, unless stated explicitly, depends only
on the details of the model \eref{e:model} and can change from line to line
even within the same block of equations.

\begin{acknowledge}
The author would like to thank Jean-Pierre Eckmann, Xue-Mei Li, Jonathan Mattingly, and Eric Vanden-Eijnden for stimulating discussions on
this and closely related problems. This work was supported by an EPSRC Advanced Research Fellowship
(grant number EP/D071593/1).
\end{acknowledge}

\section{Heuristic derivation of the main results}
\label{sec:heuristics}

In this section, we give a heuristic derivation of the results of Theorem~\ref{theo:main}. Since we are interested in the tail
behaviour of the energy in the stationary state, an important ingredient of the analysis is to isolate the `worst-case' degree
of freedom of \eref{e:model}, that would be some degree of freedom $X$ which dominates the behaviour of the energy at infinity. 
The aim of this section is to argue that it is always possible to find such a degree of freedom (but what $X$ really describes
 depends on  the details of the model, and in particular on the value of $k$) and that, for large values of $X$,
 it satisfies asymptotically an equation of the type
\begin{equ}[e:simplemodel]
dX = - \eta X^\sigma\,dt + \sqrt{2} \,dW(t)\;,
\end{equ}
for some exponent $\sigma$ and some constant $\eta>0$. Before we proceed with this programme, let us consider the
model \eref{e:simplemodel} on the set $\{X\ge 1\}$ with reflected boundary conditions at $X=1$. The invariant measure
$\mu_\star$ for \eref{e:simplemodel} then has density proportional to $\exp(-\eta X^{\sigma+1}/(\sigma+1))$ for $\sigma > -1$ and
to $X^{-\eta}$ for $\sigma = -1$. In particular, \eref{e:simplemodel} admits an invariant probability measure if and only if
$\sigma > -1$ or $\sigma = -1$ and $\eta > 1$.  
For such a model, we have the following result, which is a slight
refinement of the results obtained in \cite{Ver00:163,VerKlo04:21,Ver06}.

\begin{theorem}\label{theo:rates}
The long-time behaviour of \eref{e:simplemodel} is described by the following table:
\begin{center}
\begin{tabular}{@{\ \rule[-1.5mm]{0mm}{5mm}}cccc@{\ }}
\toprule
Parameter range & Integrability of $\mu_\star$ & Convergence speed & Prefactor\\ \midrule
$\sigma < -1$ & ---  & --- & ---\\
$\sigma = -1$, $\eta \le 1$ & --- & --- & ---\\
$\sigma = -1$, $\eta > 1$ & $X^{\eta-1 \pm \eps}$ & $t^{{1-\eta\over 2}\pm \eps}$ & $X^{\eta + 1 + \eps}$\\
$-1 < \sigma < 0$ & $\exp\bigl(\gamma_\pm X^{\sigma+1}\bigr)$ & $\exp\bigl(-\gamma_\pm t^{(1+\sigma)/(1-\sigma)}\bigr)$ & $\exp\bigl(\delta X^{\sigma + 1}\bigr)$\\
$0 \le \sigma < 1$ & $\exp\bigl(\gamma_\pm X^{\sigma+1}\bigr)$ & $\exp(-\gamma_\pm t)$ & $\exp\bigl(\delta X^{1-\sigma}\bigr)$\\
$\sigma = 1$ & $\exp\bigl(\gamma_\pm X^{\sigma+1}\bigr)$ & $\exp(-\gamma_\pm t)$ & $X^\eps$\\
$\sigma > 1$ & $\exp\bigl(\gamma_\pm X^{\sigma+1}\bigr)$ & $\exp(-\gamma_\pm t)$ & $1$\\
\bottomrule
\end{tabular}
\end{center}
The entries of this table have the same meaning as in Theorem~\ref{theo:main}, with the exception that the lower bounds on
the convergence speed toward the invariant measure hold for all $t>0$ rather than only for a subsequence.
\end{theorem}

\begin{proof}
The case $0 \le \sigma \le 1$ is very 
well-known (one can simply apply Theorem~\ref{theo:positiveAbstract3} below with either $V(X) = \exp(\delta X^{1-\sigma})$ for $\delta$ small 
enough in the case $\sigma < 1$ or with $X^\eps$ in he case $\sigma = 1$). The case $\sigma > 1$ follows from the fact that in this case one can
find a constant $C>0$ such that $\E X(1) \le C$, independently of the initial condition.

The bounds for $\sigma = -1$ and $\eta > 1$
can be found in \cite{Ver00:163,ForRob05:1565,Ver06} (a slightly weaker upper bound can also be found in \cite{RocWan01:564}).
However, as shown in \cite{BakCatGui08:727}, the upper bound can also be retrieved by using Theorem~\ref{theo:subexponential} below 
with a test function behaving like $X^{\eta+1+\eps}$ for an arbitrarily small value of $\eps$. 
The lower bound on the other hand can be obtained from Theorem~\ref{theo:lowerBoundCont} by using a test function behaving like
 $X^\alpha$, but with $\alpha \gg 1$. 
(These bounds could actually be slightly improved by choosing test functions of the form $X^{\eta+1}(\log X)^\beta$ for the upper bound
and $\exp((\log X)^\beta)$ for the lower bound.)

The upper bound for the case $\sigma \in (-1,0)$ can be found in \cite{VerKlo04:21} and more recently in \cite{DFG06SG,BakCatGui08:727}.
This and the corresponding lower bound can be obtained  similarly to above from 
Theorems \ref{theo:subexponential} and \ref{theo:lowerBoundCont}
by considering test functions of the form $\exp(a X^{\sigma+1})$ for suitable values of $a$. (Small for the upper bound and large for
the lower bound.)
\end{proof}

Returning to the problem of interest,
it was already noted in \cite{EckHai00NES,ReyTho02ECT} that $k = 1$ is a boundary between two types of completely different
behaviours for the dynamic \eref{e:model}. The remainder of this section is therefore divided into two subsections where
we analyse the behaviour of these two regimes.

\subsection{The case $\pmb{k>1}$}
\label{sec:heuristic1}

When $k > 1$, the pinning potential $V_1$ is stronger than the coupling potential $V_2$. Therefore, in this regime, 
one would expect the dynamic of the two oscillators to approximately decouple at very high energies \cite{HaiMat08:??}. This suggests that 
one should be able to find functions $H_0$ and $H_1$ describing the energies of the two oscillators such that $H_0$
is distributed approximately according to $\exp(-H_0/T)$, while the distribution of $H_1$ has heavier tails since
that oscillator is not directly damped.

In order to guess the behaviour of $H_1$ at high energies, note first that since $H_0$ is expected to have exponential tails, the regime
of interest is that where $H_1$ is very large, while $H_0$ is of order one. In this regime, the second oscillator
feels mainly its pinning potential, so that its motion is well approximated by the motion of a single free oscillator moving 
in the potential $|q|^{2k}/2k$. A simple
calculation shows that such a motion
is periodic with frequency proportional to $H_1^{{1\over 2} - {1\over 2k}}$ and with amplitude proportional to $H_1^{1\over 2k}$.  
In other words, one can find periodic functions $P$ and $Q$ such that in the regime of interest, one has (up to phases)
\begin{equ}[e:motion1]
q_1(t) \approx H_1^{1\over 2k} Q(H_1^{{1\over 2} - {1\over 2k}} t)\;,\qquad p_1(t) \approx H_1^{1\over 2} P(H_1^{{1\over 2} - {1\over 2k}}t)\;.
\end{equ}
Let now $\Psi_p$ and $\Psi_q$ be the unique solutions to $\dot \Psi_q = \Psi_p$, $\dot \Psi_p = Q$ that average out to zero over one
period. It is apparent from the equations of motion \eref{e:model} that if we assume that \eref{e:motion1} is a good model
for the dynamic of the second oscillator, then the motion of the first oscillator can, at least to lowest order, be described by
\begin{equ}[e:ptildeheur]
p_0(t) = \tilde p_0(t) - \alpha H_1^{{1\over k} - {1\over 2}}\Psi_p (H_1^{{1\over 2} - {1\over 2k}} t)\;,\qquad
q_0(t) = \tilde q_0(t) - \alpha H_1^{{3\over 2k} - {1}}\Psi_q (H_1^{{1\over 2} - {1\over 2k}} t)\;,
\end{equ}
where the functions $\tilde p_0$ and $\tilde q_0$ do not show any highly oscillatory behaviour anymore. Furthermore,
they then satisfy, at least to lowest order, the decoupled Langevin equation
\begin{equ}[e:dyneff]
d\tilde q_0 \approx \tilde p_0\,dt\;,\quad d\tilde p_0 \approx - V_1'(\tilde q_0)\,dt - \alpha \tilde q_0\,dt -\gamma \tilde p_0\,dt + \sqrt{2\gamma T}\,dw_0(t)\;,
\end{equ}
that indeed has $\exp(-H_0/T)$ as invariant measure, provided that we set
\begin{equ}
H_0 = {\tilde p_0^2 \over 2} + V_1(\tilde q_0) + {\alpha \over 2} \tilde q_0^2\;.
\end{equ}
Let us now return to the question of the behaviour of energy dissipation. The average rate of change of the total energy 
of our system is described by \eref{e:genH}. Plugging our ansatz \eref{e:ptildeheur} into this equation and using the fact that $\Psi_p$
is highly oscillatory and averages out to $0$, we obtain
\begin{equ}
\CL H \approx \gamma (T + T_\infty) - \gamma \tilde p_0^2 - \gamma \alpha^2 H_1^{{2\over k} - 1} \scal{\Psi_p^2}\;.
\end{equ}
On the other hand, it follows from \eref{e:dyneff} that one has
\begin{equ}
\CL H_0 \approx \gamma T - \gamma \tilde p_0^2\;,
\end{equ} 
so that one expects to obtain for the energy of the second oscillator the expression
\begin{equ}
\CL H_1 \approx \gamma T_\infty - \gamma \alpha^2 \scal{\Psi_p^2} H_1^{{2\over k} - 1} \;.
\end{equ}
This suggests that, at least in the regime of interest, and since the $p$-dependence of $H_1$ probably goes like ${p_1^2 \over 2}$,
the energy of the second oscillator follows a decoupled
equation of the type
\begin{equs}[e:reducedH1]
d H_1 &\approx \bigl(\gamma T_\infty - \gamma \alpha^2 \scal{\Phi_p^2} H_1^{{2\over k} - 1}\bigr)\,dt + \sqrt{2\gamma T_\infty K H_1} \, dw_1(t)\;,
\end{equs}
where $K$ is the average of $p_1^2$ over one period of the free dynamic at energy $1$, which
will be shown in \eref{e:averagep2} below to be given by $K = {2k / (1+k)}$. 

In order to analyse
\eref{e:reducedH1}, it is convenient to introduce the variable $X$ given by $X^2 = 4H_1/(\gamma T_\infty K)$, so that its evolution is 
given by 
\begin{equ}[e:evolX]
dX = \Bigl({2\over K}-1\Bigr){1\over X} - {\sqrt \gamma \alpha^2 \scal{\Phi_p^2}\over \sqrt{T_\infty K}} \Bigl({\gamma T_\infty K X^2\over 4}\Bigr)^{{2\over k}-{3\over 2}}+\sqrt{2}\,dw_1(t)\;.
\end{equ}
This shows that there is a transition at $k=2$. For $k>2$, we recover  \eref{e:simplemodel} with $\sigma = -1$ and
$\eta = 1-{2\over K} < 1$, so that one does not expect to have an invariant measure, thus recovering the corresponding statement
in Theorem~\ref{theo:main}.

At $k=2$, we still have $\sigma = -1$, but we obtain
\begin{equ}
\eta = 1-{2\over K} + {2\alpha^2 \scal{\Phi_p^2} \over T_\infty K} = {3 \over 2}{\alpha^2 \scal{\Phi_p^2} \over T_\infty}-{1\over 2} \;,
\end{equ}
so that one expects to have existence of an invariant probability measure if and only if $T_\infty < \alpha^2 \scal{\Phi_p^2}$.
Furthermore, we recover from Theorem~\ref{theo:rates} the integrability results and convergence rates of Theorem~\ref{theo:main},
noting that one has the formal correspondence $\zeta = (\eta-1)/2$. This correspondence comes from the fact that
$H \approx X^2$ in the regime of interest and that $X^{\eta-1}$ is 
the borderline for non-integrability with respect to $\mu_\star$ in Theorem~\ref{theo:rates}.

In the regime $k \in (1,2)$, the first term in the right hand side of \eref{e:evolX} is negligible, so that we
have the case $\sigma = {4\over k} - 3$.
Applying Theorem~\ref{theo:rates} then immediately allows to derive the corresponding integrability and convergence
results from Theorem~\ref{theo:main}, noting that one has the formal correspondence $\kappa = (\sigma + 1)/2$.

\subsection{The case $\pmb{k<1}$}
\label{sec:heuristic2}

This case is much more straightforward to analyse.
When $k < 1$, the coupling potential $V_2$ is stiffer than the pinning 
potential $V_1$. Therefore, one expects the two particles to behave like a single particle
moving in the potential $V_1$. This suggests that the `worst case' degree of freedom
should be the centre of mass of the system, thus motivating the change of coordinates
\begin{equ}
Q = {q_0 + q_1 \over 2}\;,\quad q = {q_1 - q_0 \over 2}\;.
\end{equ}
Fixing $Q$ and writing $y = (q,p_0,p_1)$ for the remaining coordinates, we see that there exist
matrices $A$ and $B$ and a vector $v$ such that $y$ approximately satisfies the equation
\begin{equ}
dy \approx Ay\,dt + V_1'(Q)v\,dt + B\,dw(t)\;.
\end{equ}
Here, we made the approximation $V_1'(q_0) \approx V_1'(q_1) \approx V_1'(Q)$, which is expected to
be justified in the regime of interest ($Q$ large and $y$ of order one). This shows that for $Q$ fixed,
the law of $y$ is approximately Gaussian with covariance of order one and mean proportional to $V_1'(Q)$.
Since $dQ = (p_0 + p_1)/2\,dt$, we thus expect that over sufficiently long time intervals, the dynamic of $Q$
is well approximated by
\begin{equ}
dQ \approx - C_1 V_1'(Q)\,dt + C_2\,dW(t) \approx - C_1 Q|Q|^{2k-2}\,dt + C_2\,dW(t)\;,
\end{equ}
for some positive constants $C_1$, $C_2$ and some Wiener process $W$. We are therefore reduced again to the
case of Theorem~\ref{theo:rates} with $X\propto |Q|$ and $\sigma = 2k-1$. 
Since in the regime considered here one has $H \approx X^{2k}$, this immediately allows to recover the results
of Theorem~\ref{theo:main} for the case $k < 1$.

\section{A potpourri of test function techniques}
\label{sec:Lyap}

In this section, we present the abstract results on which all the integrability and non-integrability 
results in this article are based, as well as the techniques allowing to obtain upper and lower bounds on convergence rates toward
the invariant measure. All of these results without exception are based on the existence of test functions with certain properties.
In this sense, we follow to its bitter end the Lyapunov function-based approach advocated in 
\cite{BakCatGui08:727,CatGuiWanWu07:??,CatGozGuiRob08:??} and use
it to derive not only upper bounds on convergence rates, but also lower bounds.

While most of these results from this section are known in the literature (except for the one giving the lower bounds on the convergence
of transition probabilities which appears to be new despite its relative triviality), the main interest of the present article is to provide tools
for the \textit{construction} of suitable test functions in problems where different timescales are present at
the regimes relevant for the tail behaviour of the invariant measure.

The general framework of this section is that of a Strato\-no\-vich diffusion on $\R^n$ with smooth coefficients:
\begin{equ}[e:abstract]
dx(t) = f_0(x)\,dt + \sum_{i=1}^m f_i(x)\circ dw_i(t)\;,\quad x(0) = x_0\in \R^n\;.
\end{equ}
Here, we assume that $f_j\colon \R^n \to \R^n$ are $\CC^\infty$ vector fields on $\R^n$ and the $w_i$ are independent
standard Wiener processes. Denote by $\CL$ the generator of \eref{e:abstract}, that is the differential operator
given by 
\begin{equ}
\CL = X_0 +{1\over 2} \sum_{i=1}^m X_i^2\;,\qquad X_j = f_j(x)\nabla_x\;.
\end{equ}
We make the following two standing assumptions which can easily be verified
in the context of the model presented in the introduction:

\begin{assumption}\label{ass:bound}
There exists a smooth function $H \colon \R^n \to \R_+$ with compact level sets
and a constant $C>0$ such that the bound $\CL H \le C(1 + H)$ holds.
\end{assumption}

This assumption ensures that \eref{e:abstract} has a unique global strong solution. We furthermore assume that:

\begin{assumption}\label{ass:Hor}
H\"ormander's `bracket condition' holds at every point in $\R^n$. In other words, consider the families $A_k$ (with $k \ge 0$)
of vector fields defined recursively by $A_0 = \{f_1,\ldots,f_m\}$ and
\begin{equ}
A_{k+1} = A_K \cup \{[f_j, g]\;,\quad g\in A_k\;,\quad j =0,\ldots,m\}\;.
\end{equ}
Define furthermore the subspaces $A_\infty(x) = \Span \{g(x)\,:\, \exists k>0\;\hbox{with}\; g \in A_k\}$.
We then assume that $A_\infty(x) = \R^n$ for every $x \in \R^n$.
\end{assumption}

As a consequence of H\"ormander's celebrated `sums of squares' theorem \cite{Hor67:147,Hor85},
this assumption ensures that transition probabilities for \eref{e:abstract} have smooth densities $p_t(x,y)$
with respect to Lebesgue measure. In our case, Assumption~\ref{ass:Hor} can be seen to hold because the coupling potential
is harmonic.
 
\begin{assumption}\label{ass:control}
The origin is reachable for the control problem associated to \eref{e:abstract}. That is, given any $x_0 \in \R^n$ and
any $r>0$ there exists a time $T>0$ and a smooth control $u \in \CC^\infty([0,T],\R^m)$ such that
the solution to the ordinary differential equation
\begin{equ}
{dz\over dt} = f_0(z(t)) + \sum_{i=1}^m f_i(z(t))u_i(t)\;,\quad z(0) = x_0\;,
\end{equ}
satisfies $\|z(T)\| \le r$.
\end{assumption}

The fact that Assumption~\ref{ass:control} also holds in our case is an immediate consequence of the results in \cite{EPR99EPN,Hai05:??}.
Assumptions~\ref{ass:Hor} and \ref{ass:control} taken together imply that:
\begin{claim}
\item[1.] The operator $\CL$ satisfies a strong maximum principle in the following sense. Let $D \subset \R^n$ be a compact
domain with smooth boundary such that $0 \not \in D$. Let furthermore $u \in \CC^2(D)$ be such that $\CL u(x) \le 0$ for $x$ in the interior of $D$ and $u(x) \ge 0$
for $x \in \d D$. Then, one has $u(x) \ge 0$ for all $x \in D$, see \cite[Theorem~3.2]{Bon69:277}.
\item[2.] The Markov semigroup associated to \eref{e:abstract} admits at most one 
invariant probability measure \cite{DaPZab96}. Furthermore, if such an invariant measure exists, then it has a smooth density with 
respect to Lebesgue measure.
\end{claim}

\subsection{Integrability properties of the invariant measure}

We are going to use throughout the following criterion for the existence of an invariant measure with
certain integrability properties:

\begin{theorem}\label{theo:positiveAbstract}
Consider the diffusion \eref{e:abstract} and let Assumptions~\ref{ass:Hor} and \ref{ass:control} hold. If
there exists a $\CC^2$ function $V \colon \R^n \to [1,\infty)$ such that $\limsup_{|x| \to \infty} \CL V(x) < 0$, then
there exists a unique invariant probability measure $\mu_\star$ for \eref{e:abstract}. Furthermore, $|\CL V|$ is integrable
against $\mu_\star$ and $\int \CL V(x) \mu_\star (dx) = 0$.
\end{theorem}

\begin{proof}
The proof is a continuous-time version of the results in \cite[Chapter~14]{MeyTwe93}. See also for example \cite{HaiMat08:??}.
\end{proof}

The condition given in  Theorem~\ref{theo:positiveAbstract} is actually an if and only if condition, but the other implication does not
appear at first sight to be directly useful. However, it is possible to combine the strong maximum principle with a Lyapunov-type criterion 
to rule out in certain cases the existence of a function $V$ as in Theorem~\ref{theo:positiveAbstract}.
This is the content of the next theorem which provides a constructive criterion for the non-existence of an invariant probability
measure with certain integrability properties:

\begin{theorem}\label{theo:negativeAbstract}
Consider the diffusion \eref{e:abstract} and let Assumptions~\ref{ass:bound}, \ref{ass:Hor} and \ref{ass:control} hold. Let furthermore $F\colon \R^n \to [1,\infty)$ be a continuous weight function. Assume that
there exist two $\CC^2$ functions $W_1$ and $W_2$ such that:
\begin{claim}
\item  The function $W_1$ grows in some direction, that is $\limsup_{|x| \to \infty} W_1(x) = \infty$.
\item  There exists $R>0$ such that $W_2(x) > 0$ for $|x|>R$.
\item The function $W_2$ is substantially larger than $W_1$ in the sense that there exists a positive function $H$
with $\lim_{|x| \to \infty}H(x) = +\infty$ and such that 
\begin{equ}
\limsup_{R \to \infty} {\sup_{H(x) = R} W_1(x) \over \inf_{H(x) = R} W_2(x)} = 0\;.
\end{equ}
\item There exists $R>0$ such that $\CL W_1(x) \ge 0$ and $\CL W_2(x) \le F(x)$ for $|x| > R$.
\end{claim}
Then the Markov process generated by solutions to  \eref{e:abstract} does not admit any invariant 
measure $\mu_\star$ such that $\int F(x)\,\mu_\star(dx) < \infty$.
\end{theorem}

\begin{proof}
The existence of an invariant measure that integrates $F$ is equivalent to the existence of a 
positive $\CC^2$ function $V$ such that $\CL V \le -F$ outside of some compact set \cite[Chapter~14]{MeyTwe93}.
The proof of the claim is then a straightforward extension of the proof given for the case 
$F \equiv 1$ by Wonham in \cite{Won66:195}.
\end{proof}

\begin{remark}
If one is able to choose $F\equiv 1$ in Theorem~\ref{theo:negativeAbstract}, then its conclusion is that the system under consideration
does not admit any invariant probability measure.
\end{remark}

\subsection{Convergence speed toward the invariant measure: upper bounds}

We still assume in this section that we are in the same setting as previously
and that Assumptions~\ref{ass:Hor}--\ref{ass:control} hold.
The strongest kind of convergence result that one can hope to obtain is exponential convergence toward a unique invariant measure.
In order to formulate a result of this type, given a positive function $V$, we define a weighted norm on measurable functions by
\begin{equ}
\|\phi\|_V = \sup_{x\in \R^n} {|\phi(x)| \over 1+V(x)}\;.
\end{equ}
We denote the corresponding Banach space by $\CB_b(\R^n;V)$. Furthermore, given a Mar\-kov semigroup $\CP_t$
over $\R^n$, we say that $\CP_t$ has a \textit{spectral gap} in $\CB_b(\R^n;V)$ if there exists a probability measure
$\mu_\star$ on $\R^n$ and constants $C$ and $\gamma>0$ such that the bound
\begin{equ}
\|\CP_t \phi - \mu_\star(\phi)\|_V \le Ce^{-\gamma t}\|\phi - \mu_\star(V)\|_V\;,
\end{equ}
holds for every $\phi \in \CB_b(\R^n;V)$. We will also say that a $\CC^2$ function $V\colon \R^n \to \R_+$ is a
Lyapunov function for \eref{e:abstract}  if $\lim_{|x| \to \infty} V(x) = \infty$ and
there exists a strictly positive constant $c$ such that
\begin{equ}
\CL V \le -c V\;,
\end{equ}
holds outside of some compact set.

With this notation, we have the following version  of Harris' theorem
\cite{MeyTwe93} (see also \cite{HaiMat08:??2} for an elementary proof):

\begin{theorem}\label{theo:positiveAbstract3}
Consider the diffusion \eref{e:abstract} and let Assumptions~\ref{ass:Hor} and \ref{ass:control} hold. If
there exists a Lyapunov function $V$ for \eref{e:abstract}, then $\CP_t$ admits a spectral gap in $\CB_b(\R^n;V)$.
In particular, \eref{e:abstract} admits a unique invariant measure $\mu_\star$, $\int V\,d\mu_\star < \infty$, and
convergence of transition probabilities towards $\mu_\star$ is exponential with prefactor $V$.
\end{theorem}

However, there are situations where exponential convergence does simply not take place. In such situations, one cannot hope to
be able to find a Lyapunov function as above, but it is still possible in general to find a $\phi$-Lyapunov function $V$ in the following sense.
Given a function $\phi\colon \R_+ \to \R_+$, we say that
a $\CC^2$ function $V\colon \R^n \to \R_+$ is a $\phi$-Lyapunov function if the bound
\begin{equ}
\CL V \le - \phi(V)\;,
\end{equ}
holds outside of some compact set and if $\lim_{|x|\to \infty}V(x) = \infty$.
If such a $\phi$-Lyapunov function exists, upper bounds on convergence rates toward the invariant measure can
be obtained by applying the following criterion from \cite{DFG06SG,BakCatGui08:727} (see also \cite{ForRob05:1565}):

\begin{theorem}\label{theo:subexponential}
Consider the diffusion \eref{e:abstract} and let Assumptions~\ref{ass:Hor} and \ref{ass:control} hold.
Assume that there exists a $\phi$-Lyapunov function $V$ for \eref{e:abstract}, where $\phi$ is 
some increasing smooth concave function that is strictly sublinear. Then \eref{e:abstract} admits a unique invariant measure $\mu_\star$
and there exists a positive constant $c$ such that for all $x\in \R^n$, the bound
\begin{equ}
\|\CP_t(x,\cdot) - \mu_\star\|_\TV \le c V (x) \psi(t)\;,
\end{equ} 
holds, where $\psi(t) = 1/(\phi \circ H_\phi^{-1})(t)$ and $H_\phi(t) = \int_1^t \bigl(1/\phi(s)\bigr)\,ds$.
\end{theorem}

\subsection{Convergence speed toward the invariant measure: lower bounds}

In order to obtain lower bounds on the rate of convergence towards the invariant measure $\mu_\star$,
we are going to make use of the following mechanism. Suppose that we know of some function $G$ that on the one hand
it has very heavy (non-integrable) tails under the invariant measure of some Markov process but, on the other hand,
its moments do not grow to fast. Then, this should give a lower bound on the speed of convergence towards the invariant measure
since the moment bounds prevent the process from exploring its heavy tails too quickly. This is made precise by the following 
elementary result:

\begin{theorem}\label{theo:lowerBoundCont}
Let $X_t$ be a Markov process on a Polish space $\CX$ with invariant measure $\mu_\star$ and let $G\colon \CX \to [1,\infty) $ be
such that:
\begin{claim}
\item There exists a function $f \colon [1,\infty) \to [0,1]$ such that the function $\mathrm{Id}\cdot f \colon y \mapsto y f(y)$ is increasing to infinity and such that
$\mu_\star(G \ge y) \ge f(y)$ for every $y \ge 1$.
\item There exists a function $g\colon \CX \times \R_+ \to [1,\infty)$ increasing in its second argument and such that
$ \E\bigl(G(X_t)\,|\, X_0 = x_0\bigr) \le g(x_0, t)$.
\end{claim} 
Then, one has the bound
\begin{equ}[e:lowerBoundGen]
\|\mu_{t_n} - \mu_\star \|_\TV \ge {1\over 2}f \bigl((\mathrm{Id} \cdot f)^{-1}(2g(x_0, t_n))\bigr)\;,
\end{equ}
where $\mu_t$ is the law of $X_t$ with initial condition $x_0 \in \CX$.
\end{theorem}

\begin{proof}
It follows from the definition of the total variation distance and from Chebyshev's inequality
that, for every $t \ge 0$ and every $y \ge 1$, one has the lower bound
\begin{equ}
\|\mu_t - \mu_\star \|_\TV \ge \mu_\star(G(x) \ge y) - \mu_t(G(x) \ge y) \ge f(y) - {g(x_0,t)\over y}\;.
\end{equ}
Choosing $y$ to be the unique solution to the equation $y f(y) = 2g(x_0,t)$, the result follows.
\end{proof}

The problem is that in our case, we do not in general have sufficiently good information on the tail behaviour of $\mu_\star$
to be able to apply Theorem~\ref{theo:lowerBoundCont} as it stands. However, it follows immediately from the proof that the
bound \eref{e:lowerBoundGen} still holds for a subsequence of times $t_n$ converging to $\infty$, provided that the bound
$\mu_\star(G \ge y_n) \ge f(y_n)$ holds for a sequence $y_n$ converging to infinity.
This observation allows to obtain the following corollary that is of more use to us:

\begin{corollary}\label{theo:lowerBound}
Let $X_t$ be a Markov process on a Polish space $\CX$ with invariant measure $\mu_\star$ and let $W\colon \CX \to [1,\infty)$ be
such that $\int W(x)\,\mu_\star(dx) = \infty$.
Assume that there exist $F\colon [1,\infty) \to \R$ and $h\colon [1,\infty) \to \R$ such that:
\begin{claim}
\item $h$ is decreasing and $\int_1^\infty h(s)\,ds <\infty$.
\item $F\cdot h$ is increasing and $\lim_{s\to \infty} F(s) h(s) = \infty$.
\item There exists a function $g\colon \CX \times \R_+ \to \R_+$ increasing in its second argument and such that
$ \E\bigl( (F\circ W)(X_t)\,|\, X_0 = x_0\bigr) \le g(x_0, t)$.
\end{claim} 
Then, for every $x_0 \in \CX$, there exists a sequence of times $t_n$ increasing to infinity
such that the bound
\begin{equ}
\|\mu_{t_n} - \mu_\star \|_\TV \ge h \bigl((F\cdot h)^{-1}(g(x_0, t_n))\bigr)
\end{equ}
holds, where $\mu_t$ is the law of $X_t$ with initial condition $x_0 \in \CX$.
\end{corollary}

\begin{proof}
Since $\int W(x)\,\mu_\star(dx) = \infty$, there exists a sequence $w_n$ increasing to infinity such that
$\mu_\star(W(x) \ge w_n) \ge 2h(w_n)$, for otherwise we would have the bound
\begin{equ}
\int W(x)\,\mu_\star(dx) = 1 + \int_1^\infty  \mu_\star(W(x) \ge w)\,dw \le 1 + 2\int_1^\infty h(w)\,dw < \infty\;,
\end{equ}
thus leading to a contradiction. Applying Theorem~\ref{theo:lowerBoundCont} with $G = F \circ W$
and $f = 2h \circ F^{-1}$ concludes the proof.
\end{proof}

\section{Existence and non-existence of an invariant probability measure}
\label{sec:existNonExist}

\subsection{Non-existence of an invariant measure}

The aim of this section is to show that \eref{e:model} does not admit any invariant probability measure if $k > 2$ or $k=2$ and
$T_\infty > \alpha^2 \hat C$.  Note first that one has an upper bound on the evolution of the total energy of the system given by
\begin{equ}
\CL H = \gamma T + \gamma T_\infty -\gamma p_0^2\;,
\end{equ}
which suggests that $H$ is a natural choice
for the function $W_2$ in Wonham's criterion for the non-existence of an invariant probability measure.

It therefore remains to find a function $W_1$ that grows
to infinity in some direction (not necessarily all), that is dominated by the energy in the sense that
\begin{equ}[e:dominated]
\lim_{E \to \infty}  {1\over E} \sup_{H(p,q) = E} W_1(p,q) = 0\;,
\end{equ}
and such that $\CL W_1 \ge 0$
outside of some compact region $\CK$. 

In order to construct $W_1$, we use some of the ideas introduced in \cite{HaiMat08:??}. The technique used
there was to make a change of variables such that, in the new variables, the motion of the `fast' oscillator
decouples from that of the `slow oscillator'. In the situation at hand, we wish to show that the energy of the 
second oscillator grows, so that the relevant regime is the one where that energy is very high. 

One is then tempted to set
\begin{equ}[e:defW_1]
W_1 = H^{-\zeta}(H - \HU)\;,
\end{equ}
for some (typically small) exponent $\zeta \in (0,1)$, 
where $\HU$ is a multiple of the energy of the first oscillator, expressed in the `right' set of 
variables. In order to compute $\CL W_1$, we make use of the following `chain rule' for $\CL$:
\begin{equ}[e:chainrule]
\CL(f\circ g) = (\d_i f \circ g) \CL g_i + (\d_{ij}^2 f \circ g)\Gamma(g_i,g_j)\;,
\end{equ}
(summation over repeated indices is implied), where we defined the `carr\'e du champ' operator
\begin{equ}
\Gamma(g_i, g_j) = \gamma T\d_{p_0}g_i\d_{p_0}g_j + \gamma T_\infty \d_{p_1}g_i \d_{p_1}g_j\;.
\end{equ}
(Note that it differs by a factor two from the usual definition in order to keep expressions as compact as possible.) 
This allows us to obtain the identity
\begin{equs}[e:LW_1]
 \CL W_1 &= H^{-\zeta} (\gamma T + \gamma T_\infty -\gamma p_0^2 - \CL \HU) \\
&\quad - \gamma \zeta H^{-\zeta - 1}(H-\HU) (T + T_\infty - p_0^2) \\
&\quad - 2\gamma \zeta H^{-\zeta - 1}\bigl(T p_0(p_0 - \d_{p_0} \HU)
+ T_\infty p_1(p_1 - \d_{p_1} \HU)\bigr) \\
&\quad + \gamma \zeta (\zeta+1) H^{-\zeta -2} (H-\HU) \bigl(T p_0^2
+ T_\infty p_1^2\bigr)
\;.
\end{equs}
Following our heuristic calculation in Section~\ref{sec:heuristic1}, we expect that at high energies, one has
$\CL \HU \approx \gamma T - \gamma \tilde p_0^2$, were $\tilde p_0$ denotes the correct variable
in which to express the motion of the oscillator. One would then like to first choose our compact set $\CK$ sufficiently large
so that the expression on the first line of \eref{e:LW_1} is larger than $\delta H^{-\zeta}(1 + p_0^2)$ for some $\delta > 0$. Then,
by choosing $\zeta$ sufficiently close to zero, one would like to make the remaining terms
sufficiently small so that $\CL W_1 > 0$ outside of a compact set. This is made precise by the following lemma:
\begin{lemma}\label{lem:choiceH0}
If there exist a $\CC^2$ function $\HU \colon \R^4 \to \R$ and strictly positive constants $c$ and $C$ such that, outside of some compact subset of $\R^4$,
it satisfies the bounds
\begin{equ}
\CL \HU \le \gamma (T+ T_\infty - p_0^2) - c(1+p_0^2)\;,\quad |\HU| +  |\d_{p_0}\HU|^2 + |\d_{p_1}\HU|^2 \le C H\;,
\end{equ}
and such that furthermore
\begin{equ}[e:behaveinfty]
\limsup_{E \to \infty} {1\over E}\inf_{H(x) = E} \HU(x) < 1\;,
\end{equ}
then  \eref{e:model} admits no invariant probability measure.
\end{lemma}

\begin{proof}
Setting $W_1$ as in \eref{e:defW_1}, we see from \eref{e:LW_1}
and the assumptions on $\HU$  that there exists a constant $C>0$  independent of $\zeta \in (0,1)$
such that the bound
\begin{equ}
\CL W_1 \ge cH^{-\zeta} (1 + p_0^2) - \zeta C H^{-\zeta}(1+ p_0^2)
\end{equ}
holds outside of some compact set. Choosing $\zeta < c/C$, it follows that $\CL W_1 > 0$ outside of some 
compact subset of $\R^4$. Assumption \eref{e:behaveinfty} makes sure that $W_1$ grows to $+\infty$ in some
direction and rules out the trivial choice $\HU \propto H$.
Since it follows furthermore from the assumptions that $W_1 \le CH^{1-\zeta}$, \eref{e:dominated} holds 
so that the assumptions of Wonham's criterion are satisfied.
\end{proof}

The remainder of this section is devoted to the construction of such a function $\HU$, thus giving rise to the following result:

\begin{theorem}\label{theo:negative}
There exists a constant $\hat C$ such that, if either $k > 2$, or $k = 2$ and $T_\infty > \alpha^2 \hat C$, the model \eref{e:model} admits no invariant probability measure.
\end{theorem}

\begin{remark}
As will be seen from the construction, the constant $\hat C$ is really equal to the constant $\scal{\Phi_p^2}$ from Section~\ref{sec:heuristic1}.
\end{remark}

\begin{proof}
As in \cite{HaiMat08:??}, we define the Hamiltonian
\begin{equ}
\Hf(P,Q) = {P^2 \over 2} + {|Q|^{2k} \over 2k}
\end{equ}
 of a `free' oscillator on $\R^2$ and its  generator
\begin{equ}[e:defL0]
\CL_0 = P\d_Q - Q|Q|^{2k-1}\d_P\;.
\end{equ}
These definitions will be used for all of the remainder of this article, except for Section~\ref{sec:smallk}. The variables $(P,Q)$ 
should be thought of as `dummy variables' that will be replaced by for example $(p_1,q_1)$ or $(p_0,q_0)$ when needed.

We also define $\Phi$ as the unique centred\footnote{We say that a function on $\R^2$ is centred if it averages to $0$ along orbits of the
Hamiltonian system with Hamiltonian $\Hf$.} solution to the Poisson equation
\begin{equ}
\CL_0 \Phi = Q - \CR(P,Q)\;,
\end{equ}
where $\CR \colon \R^2 \to \R$ is a smooth function averaging out to zero on level sets of $\Hf$, and
such that $\CR = 0$ outside of a compact set and $\CR = Q$ inside an open set containing the origin.
The reason for introducing the correction term $\CR$ is so that the function $\Phi$ is smooth everywhere 
including the origin, which would
not be the case otherwise. It follows from \cite[Prop.~3.7]{HaiMat08:??} that $\Phi$ scales like $\Hf^{{1\over k}-{1\over 2}}$
in the sense that, outside a compact set, it can be written as $\Phi = \Hf^{{1\over k}-{1\over 2}} \Phi_0(\omega)$,
where $\omega$ is the angle variable conjugate to $\Hf$.

Inspired by the formal calculation from Section~\ref{sec:heuristic1}, we 
then define $\tilde p_0 = p_0 - \alpha \Phi(p_1, q_1)$, so that the equations of motion for the first oscillator turn into
\begin{equs}
d q_0 &= \tilde p_0\,dt  + \alpha \Phi\,dt \label{e:dyntilde}\\
d\tilde p_0 &= - q_0| q_0|^{2k-2}\,dt - \alpha q_0\,dt - \gamma\,p_0\,dt + \sqrt{2\gamma T}\,dw_0(t) \\
&\quad + \alpha \CR\,dt - \alpha^2 (q_0 - q_1) \d_P \Phi\, dt - \alpha  \sqrt{2\gamma T_\infty} \d_P \Phi\, dw_1(t) - \alpha \gamma T_\infty \d_P^2 \Phi\, dt\\
&\quad  + \alpha R_1'(q_1) \d_P \Phi\,dt - R_1'(q_0)\,dt\;.
\end{equs}
Here, we omitted the argument $(p_1, q_1)$ from $\Phi$, its partial derivatives, and $\CR$ in order to make the expressions shorter. Setting 
\begin{equ}[e:defH0hat]
\tilde H_0 = {\tilde p_0^2 \over 2} + \VV(q_0) + \ps \tilde p_0 q_0\;,
\qquad \VV(q) = V_1(q) + \alpha {q^2 \over 2}\;,
\end{equ}
we obtain the following identity:
\begin{equs}
\CL \tilde H_0 &= \gamma T - (\gamma - \ps) p_0^2  - \ps |q_0|^{2k} - \alpha \ps |q_0|^2 - \gamma \ps p_0 q_0 \\
&\quad + \alpha^2 (\gamma-\ps) \Phi^2 + \alpha (\gamma - \ps) \Phi \tilde p_0+ \alpha \Phi \VV'(q_0)\label{e:LH0}\\
&\quad + \alpha \tilde p_0 \CR - \alpha^2 \tilde p_0 (q_0 - q_1)\d_P \Phi 
 -\alpha \gamma T_\infty \tilde p_0 \d_P^2 \Phi + \alpha^2 \gamma T_\infty (\d_P \Phi)^2 \\
&\quad + \ps q_0 \bigl(\alpha \CR - \alpha^2 (q_0 - q_1) \d_P \Phi - \alpha \gamma T_\infty \d_P^2 \Phi\bigr) \\
&\quad + (\tilde p_0 + \theta q_0) \alpha R_1'(q_1) \d_P \Phi
\end{equs}
All the terms on lines 3 to 5 (and also the terms on line 2 provided that $k > 2$)
 are of the form $f(p_0, q_0)g(p_1, q_1)$ with $g$ a function going to $0$
at infinity and $f$ a function such that $f(p_0,q_0)/\Hf(p_0,q_0)$ goes to $0$ at infinity. It follows that, for every $\eps > 0$,
there exists a compact set $K_\eps \subset \R^4$ such that, outside of $K_\eps$, one has the inequality
\begin{equs}[e:boundH0orig]
\CL \tilde H_0 &\le \gamma T -\gamma p_0^2 + \eps + \Bigl(\ps + {\ps \gamma^2 \over 4\alpha}+\eps\Bigr) p_0^2  - (\ps-\eps) |q_0|^{2k} \\
&\quad + \alpha^2 (\gamma-\ps) \Phi^2 + \alpha (\gamma - \ps) \Phi \tilde p_0+ \alpha \Phi \VV'(q_0)\;.
\end{equs}
Here, we also used the fact that $\gamma \ps p_0 q_0 \le \alpha \ps |q_0|^2 +  {\gamma^2\ps \over 4\alpha}p_0^2$. If $k > 2$, then the function $\Phi$ also converges to $0$ at infinity, so that the bound
\begin{equs}
\CL \tilde H_0 &\le \gamma T -\gamma p_0^2 + \eps + \Bigl(\ps + {\ps \gamma^2 \over 4\alpha}+\eps\Bigr) p_0^2  - (\ps-\eps) |q_0|^{2k}\;,
\end{equs}
holds outside of a sufficiently large compact set. It follows that 
the conditions of Lemma~\ref{lem:choiceH0} are satisfied by $\HU = (1+\delta)\tilde H_0$ for 
$\delta>0$ sufficiently small whenever $T_\infty > 0$, provided that one also chooses both $\ps$ and $\eps$ 
sufficiently small.

The case $k = 2$ is slightly more subtle and we assume that $k=2$ for the remainder of this proof. In particular, this implies that $\Phi$ scales like a constant
outside of some compact set. This suggests that the term $\Phi^2$ should average out to a constant, whereas the terms
$\Phi \tilde p_0$ and $\Phi \VV'(q_0)$ should average out to zero, modulo some lower-order corrections. It turns out that these
corrections will have the unfortunate property that they grow faster than $\Hf$ in the $(p_0, q_0)$ variables. On the other hand, 
we notice that both $\tilde p_0$ and $\VV'(q_0)$ do grow slower than $\Hf$ at infinity. As a consequence, it is sufficient to compensate
these terms for `low' values of $(p_0, q_0)$.

Before giving the precise expression for a function $\HU$ that satisfies the assumptions of Lemma~\ref{lem:choiceH0} 
for the case $k=2$, we make some preliminary calculations.
We denote by $\psi \colon \R \to \R_+$
a smooth decreasing `cutoff function' such that $\psi(x) = 1$ for $x \le 1$ and  $\psi(x) = 0$ for $x \ge 2$. Given a positive constant
$E$, we also set
\begin{equ}
\psi_E(\tilde p_0, q_0) = \psi \Bigl({\Hf(\tilde p_0,q_0) \over E}\Bigr)\;,\qquad
\psi_E' = {1\over E}\psi' \Bigl({\Hf \over E}\Bigr)\;,\qquad
\psi_E'' = {1\over E^2}\psi'' \Bigl({\Hf \over E}\Bigr)\;.
\end{equ}

\begin{definition}
We will say that a function $f\colon \R_+ \times \R^4 \to \R$ is \textit{negligible} if, for every $\eps > 0$,
there exists  $E_\eps > 0$ and, for every $E > E_\eps$ there exists a compact set $K_{E,\eps} \Subset \R^4$ such that
the bound $|f(E;p,q)| \le \eps \bigl(1 + \Hf(\tilde p_0,q_0)\bigr)$ holds for
 every $(p,q) \not\in K_{E,\eps}$.
\end{definition}

With this definition at hand, we introduce the notations
\begin{equ}[e:relnegl]
f \lesssim g\;,\qquad f \sim g\;,
\end{equ}
to mean that there exists a negligible function $h$ such that $f \le g + h$ or $f = g + h$ respectively. With this notation, we can rewrite \eref{e:boundH0orig} as
\begin{equs}[e:boundLH0]
\CL \tilde H_0 \lesssim \gamma T - \gamma_\ps  p_0^2  - \ps |q_0|^{2k} 
 + \alpha^2 (\gamma-\ps) \Phi^2 + f_\ps \Phi \;,
\end{equs}
where we introduced the constant $\gamma_\ps = \gamma - \ps(1+ {\gamma^2 \over 4\alpha})$ and the function $f_\ps =  \alpha (\gamma - \ps) \tilde p_0+ \alpha \VV'(q_0)$.
 
\begin{lemma}\label{lem:negligible}
Let $a,b \ge 0$ and let $f,g\colon \R^2 \to \R$ be functions that scale like $\Hf^a$ and $\Hf^{-b}$ respectively. Then, the following functions
are negligible:
\begin{claim}
\item[i)] $f(\tilde p_0, q_0)g(p_1, q_1) \psi_E(\tilde p_0,q_0)$, provided that $b > 0$.
\item[ii)] $f(\tilde p_0, q_0)g(p_1, q_1) \psi_E'(\tilde p_0,q_0)$, provided that $b > 0$ or $a < 2$.
\item[ii')] $f(\tilde p_0, q_0)g(p_1, q_1) \bigl(\psi_E'(\tilde p_0,q_0)\bigr)^2$, provided that $b > 0$ or $a < 3$.
\item[iii)] $f(\tilde p_0, q_0)g(p_1, q_1) \psi_E''(\tilde p_0,q_0)$, provided that $b > 0$ or $a < 3$.
\item[iv)] $f(\tilde p_0, q_0)g(p_1, q_1) \bigl(1-\psi_E(\tilde p_0,q_0)\bigr)$, provided that $a < 1$.
\end{claim}
\end{lemma}

\begin{proof}
We assume without loss of generality that the bounds $f(p,q) \le 1\vee \Hf^a(p,q)$ and $g(p,q) \le 1\wedge \Hf^{-b}(p,q)$ hold for every $(p,q) \in \R^2$.

In the case i), we take $E_\eps = 1$ and choose for $K_{E,\eps}$ the set of points such that either $\Hf(\tilde p_0, q_0) \ge 2E$,
in which case the expression vanishes, or $\Hf(p_1,q_1) \ge (2E)^{a/b} \eps^{-1/b}$ in which case the expression is smaller than $\eps$.

The case ii) with $b>0$ follows exactly like the case i), so we consider the case $a<2$ and $b=0$. Since $\psi_E' = 0$
if $\Hf(\tilde p_0, q_0) \ge 2E$ and is smaller than $1/E$ otherwise, we have the bounds
\begin{equ}
|f(\tilde p_0, q_0)g(p_1, q_1) \psi_E'(\tilde p_0,q_0)| \le (1+\Hf(\tilde p_0,q_0)) E^{(0 \vee a-1) -1}\;.
\end{equ}
Since the exponent of $E$ appearing in this expression 
is negative provided that $a < 2$, this is shown to be negligible by choosing $E_\eps$
sufficiently large and setting $K_{E,\eps} = \emptyset$. Cases ii') and iii) follow in a nearly identical manner.

In the case iv), we use the fact that since $a < 1$, for fixed $\eps > 0$, we can find a constant $C_\eps$ such that $|f(\tilde p_0, q_0)| \le {\eps\over 2} \Hf (\tilde p_0,q_0)+ C_\eps$. 
We then set $E_\eps = 2C_\eps / \eps$, so that $\Hf(\tilde p_0, q_0) \ge E_\eps$ implies $\Hf(\tilde p_0, q_0) \ge {2C_\eps \over \eps}$.
Since $g$ is bounded by $1$ by assumption and since $1-\psi_E$ vanishes for $\Hf(\tilde p_0, q_0) \le E$, it follows that the expression
iv) is uniformly bounded by $\eps \Hf(\tilde p_0,q_0)$ for $E \ge E_\eps$.
\end{proof}

\begin{remark}
In the case where both $b >0$ and $a < 1$, the function $f(\tilde p_0, q_0)g(p_1, q_1)$ is negligible, which can be seen
from cases i) and ii) above.
\end{remark}
 
\begin{corollary}\label{cor:neglderpsiE}
In the setting of Lemma~\ref{lem:negligible}, the following functions
are negligible:
\begin{claim}
\item[v)] $f(\tilde p_0, q_0)g(p_1, q_1) \d_{p_0}\psi_E(\tilde p_0,q_0)$ provided that $b>0$ or $a < 3/2$.
\item[vi)] $f(\tilde p_0, q_0)g(p_1, q_1) \d_{p_1}\psi_E(\tilde p_0,q_0)$
\item[vii)] $f(\tilde p_0, q_0)g(p_1, q_1) \CL\psi_E(\tilde p_0,q_0)$ provided that $b>{1\over 2} - {1\over k}$ or $b = {1\over 2} - {1\over k}$ and $a < 1$.
\end{claim}
\end{corollary}

\begin{proof}
We can write
\begin{equs}
f(\tilde p_0, q_0)g(p_1, q_1) \d_{p_0}\psi_E(\tilde p_0,q_0) &= \tilde p_0 f(\tilde p_0, q_0)g(p_1, q_1) \psi_E'\;, \\
f(\tilde p_0, q_0)g(p_1, q_1) \d_{p_1}\psi_E(\tilde p_0,q_0) &= \tilde p_0 f(\tilde p_0, q_0)g(p_1, q_1) \d_P \Phi(p_1,q_1)\psi_E'\;,
\end{equs}
so that the first two cases can be reduced to case ii) of Lemma~\ref{lem:negligible}. For case vii), we use the fact that
\begin{equ}[e:exprLpsi]
\CL \psi_E = \psi_E' \CL \Hf + \gamma  \bigl(T_0 + T_\infty (\d_P\Phi)^2\bigr) \tilde p_0^2\psi_E''\;,
\end{equ}
and that $\CL \Hf$ consists of terms that all scale like $\Hf^c(\tilde p_0, q_0)\Hf^d(p_1, q_1)$ with $c \le 1$ and $d \le {1\over k} - {1\over 2}$
(see \eref{e:LH0}) to reduce ourselves to
cases ii) and iii) of Lemma~\ref{lem:negligible}.
\end{proof}

Before we proceed with the proof of Theorem~\ref{theo:negative}, we state two further preliminary results
that will turn out to be useful also for the analysis of the case $k \in (1,2)$:

\begin{lemma}\label{lem:scalingLf}
Let $k \in (1,2]$ and let $f\colon \R^2 \to \R$ be a function that scales like $\Hf^a$ for some $a \in \R$. Then, the function $g = \CL \bigl(f(\tilde p_0, q_0)\bigr)$ consists of terms that are bounded by multiples of
$\Hf^c(\tilde p_0, q_0)\Hf^d(p_1, q_1)$ with either
$c \le a + {1\over 2} - {1\over 2k}$ and $d \le 0$ or $c \le a -{1\over 2}$ and $d \le {1\over k} - {1\over 2}$. 
\end{lemma}

\begin{proof}
It follows from \eref{e:dyntilde} that
\begin{equs}
g &= \bigl(-\alpha q_0  - \gamma (\tilde p_0 + \alpha \Phi) + \alpha \CR - \alpha^2 (q_0 - q_1)\d_P \Phi
- \alpha \gamma T_\infty \d_P^2 \Phi\bigr) \d_P f \\
&\quad +\bigl(\alpha R_1'(q_1) \d_P \Phi - R_1'(q_0)\bigr)  \d_P f \\
&\quad + \gamma\bigl(T + T_\infty (\d_P \Phi)^2\bigr)\d_P^2 f + (\tilde p_0 + \alpha \Phi) \d_Q f - q_0|q_0|^{2k-2} \d_P f\;,
\end{equs}
from which the claim follows by simple powercounting.
\end{proof}

\begin{lemma}\label{lem:scalingLf2}
Let $k \in (1,2]$ and let  $f\colon \R^2 \to \R$ be a function that scales like $\Hf^{-b}$ for some $b \in \R$. Then, the function
$g = \CL \bigl(f(p_1, q_1)\bigr)  - (\CL_0 f)(p_1, q_1)$ consists of terms that are bounded by multiples of 
$\Hf^c(\tilde p_0, q_0)\Hf^d(p_1, q_1)$ with either $c \le {1\over 2k}$ and $d \le -b - {1\over 2}$
or  $c \le 0$ and $d \le -b - {1\over 2} + {1\over 2k}$.
\end{lemma}

\begin{proof}
It follows from \eref{e:model} that
\begin{equ}[e:exprg]
g = \alpha (q_0 - q_1) \d_P f - R_1'(q_1)\d_P f + \gamma T_\infty \d_P^2 f \;,
\end{equ}
from which the claim follows.
\end{proof}

We now return to the proof of Theorem~\ref{theo:negative}.
We define $\Psi$ as the unique centred solution to the equation $\CL_0 \Psi = \Phi$. One can see in a similar way
as before that $\Psi$ scales like $\Hf^{-{1\over 4}}$. Since $\Phi$ scales like a constant, there exists some constant $\hat C$
such that $\Phi^2$ averages to $\hat C$ outside a compact set. While the constant $\hat C$ can not be expressed in simple terms,
it is easy to compute it numerically: $\hat C \approx 0.6354699$\footnote{All displayed digits are accurate.}.

In particular, there exists a function $\hat \CR\colon \R_+ \to \R_+$
with compact support and such that $\Phi^2 - \hat C + \hat \CR(\Hf(P,Q))$ is centred. Denote by $\Xi$ the centred solution to the equation
\begin{equ}[e:defXi]
\CL_0 \Xi = \Phi^2 - \hat C + \hat \CR(\Hf(P,Q))\;,
\end{equ}
so that $\Xi$ scales like $\Hf^{-{1\over 4}}$, just like $\Psi$ does. With these definitions at hand, we set
\begin{equ}[e:defH0]
H_0 = \tilde H_0 - \bigl(\alpha^2(\gamma - \ps) \Xi(p_1,q_1) + f_\ps  \Psi(p_1,q_1)\bigr)\psi_E(\tilde p_0, q_0)\;,
\end{equ}
where we used the function $f_\ps$ introduced in \eref{e:boundLH0}. 
Recalling that $f_\ps$ consists of terms scaling like $\Hf^a(\tilde p_0, q_0)$ with $a \le {3\over4}$, we
obtain  from Lemmas~\ref{lem:scalingLf2} and \ref{lem:negligible} that
\begin{equ}
f_\ps  \CL(\Psi(p_1,q_1)) \psi_E =  f_\ps \Phi - f_\ps \Phi (1- \psi_E) + f_\ps (\CL \Psi - \CL_0 \Psi) \psi_E
\sim f_\ps \Phi\;.
\end{equ}
Similarly, we obtain that
\begin{equs}
\CL(\Xi(p_1,q_1)) \psi_E = \Phi^2 - \hat C + \bigl(\Phi^2 - \hat C\bigr)(1- \psi_E) + \hat \CR \psi_E
\sim \Phi^2 - \hat C\;.
\end{equs}
It therefore follows from \eref{e:boundLH0}, the facts that $\d_{p_0} \tilde p_0 = 1$ and $\d_{p_1} \tilde p_0 = -\alpha \d_P \Phi(p_1,q_1)$, and the multiplication rule for $\CL$, that one has
the bound
\begin{equs}
\CL H_0 &\lesssim  \gamma T - \gamma_\ps  p_0^2  - \ps |q_0|^{2k} 
 + \alpha^2 (\gamma-\ps) \hat C \\
 &\quad - (\alpha^2(\gamma - \ps)\Xi + f_\ps \Psi) \CL \psi_E  - \CL f_\ps \Psi \psi_E \\
 &\quad + C| \d_P \Xi \d_{p_1}\psi_E | + C |f_\ps \d_P \Psi \d_{p_1} \psi_E| \\
 &\quad + C |\Psi \d_P f_\ps \bigl(1+(\d_P \Phi)^2\bigr) \psi_E'| + C |\d_P \Psi \d_P f_\ps \d_P \Phi \psi_E|\;.
\end{equs}
The terms on the second and third line are negligible by Lemma~\ref{lem:scalingLf} and Corollary~\ref{cor:neglderpsiE}.
The terms on the last line are negligible by Lemma~\ref{lem:negligible}, so that we finally obtain the bound
\begin{equ}[e:boundH0]
\CL H_0 \lesssim  \gamma (T + \alpha^2 \hat C) - \gamma_\ps  p_0^2  - \ps |q_0|^{2k}  - \alpha^2 \ps \hat C\;.
\end{equ}
Since the constant $\gamma_\ps$ can be made arbitrarily close to $\gamma$ by choosing $\ps$ sufficiently
small, we see as before that, provided that $T_\infty > \alpha^2 \hat C$, it is possible to choose $\ps$ small enough and $E$ large enough so 
that the choice $\HU = (1+\delta)H_0$ with $\delta > 0$ sufficiently small again allows to satisfy the conditions of Lemma~\ref{lem:choiceH0}.
This concludes the proof of Theorem~\ref{theo:negative}.
\end{proof}

\subsection{Existence of an invariant measure}

Theorem~\ref{theo:negative} has the following converse:

\begin{theorem}\label{theo:positive}
If either $1 < k < 2$, or $k = 2$ and $T_\infty < \alpha^2 \hat C$, the model \eref{e:model} admits a unique invariant 
probability measure $\mu_\star$. The constant $\hat C$ is the same as in Theorem~\ref{theo:negative}.
\end{theorem}

\begin{proof}
Somewhat surprisingly given that the two statements are almost diametrically opposite, it is possible to 
prove this positive result in very similar way to the previous negative result
by constructing the right kind of Lyapunov function. As before, the case $k=2$ will be treated somewhat differently.

\textbf{The case $\pmb{k=2}$.} Similarly to what we did in \eref{e:defW_1}, the 
idea is to look at the function $V = H - cH_0$ for a suitable constant $c$, but this time we choose it in such a way that 
$\lim_{|(p,q)| \to \infty} V = \infty$ and $\limsup_{|(p,q)| \to \infty} \CL V < 0$, so that we can apply Theorem~\ref{theo:positiveAbstract}.
Note that, with the same notations as in the proof of Theorem~\ref{theo:negative}, one has from \eref{e:LH0} 
\begin{equs}
\CL \tilde H_0 &\sim \gamma T - (\gamma - \ps) p_0^2  - \ps |q_0|^{2k} - \alpha \ps |q_0|^2 - \gamma \ps p_0 q_0 \\
&\quad + \alpha^2 (\gamma-\ps) \Phi^2 + f_\ps \Phi\;,
\end{equs}
so that, provided this time that we choose $\ps < 0$ in the definition of $\tilde H_0$ (and therefore of $H_0$), we have the bound
\begin{equs}
\CL H_0 &\sim \gamma T + \alpha^2(\gamma-\ps)\hat C - (\gamma - \ps) p_0^2  - \ps |q_0|^{2k} - \alpha \ps |q_0|^2 - \gamma \ps p_0 q_0 \\
&\gtrsim  \gamma T + \alpha^2(\gamma-\ps)\hat C - \gamma_\ps p_0^2  - \ps |q_0|^{2k}\;,
\end{equs}
where we set $\gamma_\ps = \gamma - \ps(1+{\gamma^2 \over 4\alpha})$ as before. Here, the function
$H_0$ is as in \eref{e:defH0}  and depends on a large parameter $E$ as above.
If we choose $c < 1$, the function 
\begin{equ}[e:defV2]
V = H - cH_0
\end{equ}
does then indeed grow to infinity in all directions and we have
\begin{equ}
\CL V \lesssim \gamma T(1-c) + \gamma T_\infty - c\alpha^2(\gamma-\ps)\hat C - c|\ps| |q_0|^{2k} - (\gamma - c\gamma_\ps) |p_0|^2\;.
\end{equ}
If the assumption $\alpha^2 \hat C > T_\infty$ is satisfied, we can find a constant $\beta > 0$ such that
\begin{equ}
\gamma T(1-c) + \gamma T_\infty - c\alpha^2(\gamma-\ps)\hat C \le -\beta
\end{equ}
for all $\ps$ sufficiently small and all $c$ sufficiently close to $1$. By fixing $c$ and making $\ps$ sufficiently
small, we can furthermore ensure that $\gamma - c\gamma_\ps > 0$. This shows that, by first choosing $c$
sufficiently close to $1$, then making $\ps$ very small and finally choosing $E$ very large, we have constructed
a function $V$ satisfying the assumptions of Theorem~\ref{theo:positiveAbstract}, thus concluding the proof in the case $k=2$.

\textbf{The case $\pmb{k<2}$.} Even though one would expect this to be the easier case, it turns out to
be tricky because of the fact that the approximate decoupling of the oscillators at high energies is not such a good description
of the dynamic anymore. 
The idea is to consider again the variable $\tilde p_0$ introduced previously
but, because of the fact that the function $\Phi$ is now no longer bounded, we are going to multiply certain
correction terms by a `cutoff function'.

Since we are following a similar line of proof to the non-existence result and since we expect from \eref{e:reducedH1} to be able
to find  a function $V$ close to $H$ and such that it asymptotically satisfies a bound
of the type $\CL V \approx - \Hf(\tilde p_0,q_0) - \Hf^{{2\over k}-1}(p_1,q_1)$, this suggests that we should introduce
the following notion of a negligible function suited to this particular case:
\begin{definition}\label{def:Negligible2}
A function $f\colon \R^4 \to \R$ is \textit{negligible} if, for every $\eps > 0$,
there exists a compact set $K_\eps$ such that the bound $|f(p,q)| \le \eps \bigl(\Hf(\tilde p_0,q_0) + \Hf^{{2\over k}-1}(p_1,q_1)\bigr)$ holds for
 every $(p,q) \not\in K_\eps$.
\end{definition}

We also introduce the notations $\sim$ and $\lesssim$ similarly to before.
For $\ps > 0$, we then set $\hat V = H + \ps \tilde p_0 q_0$, so that \eref{e:dyntilde} yields
\begin{equs}
\CL \hat V &= \gamma (T + T_\infty) - \gamma p_0^2 + \alpha \ps \tilde p_0 \Phi
+ \ps \tilde p_0^2 - \ps |q_0|^{2k}   - \alpha\ps |q_0|^2 - \gamma \ps p_0 q_0 \\
&\quad +\ps \alpha q_0 \bigl(\CR - \alpha(q_0-q_1)\d_P\Phi - \gamma T_\infty \d_P^2\Phi\bigr) \\
&\quad + \ps q_0 \bigl(\alpha R_1'(q_1) \d_P \Phi - R_1'(q_0)\bigr)\;. 
\end{equs}
It is straightforward to check that all of the terms on the second and third lines are negligible.
Using the definition of $\tilde p_0$ and completing the square for the term $\alpha |q_0|^2 + \gamma p_0 q_0$,
we thus obtain the bound
\begin{equ}
\CL \hat V \lesssim  - \gamma_\ps  \tilde p_0^2 - \ps |q_0|^{2k} +  c_\ps \tilde p_0 \Phi - \alpha_\ps \Phi^2 \;.
\label{e:boundVhat}
\end{equ}
Here, we defined the constants
\begin{equ}
\alpha_\ps \eqdef \alpha \gamma \Bigl(\alpha - {\gamma \ps \over 4}\Bigr) \;,\qquad
c_\ps \eqdef (\alpha \ps - 2\alpha \gamma + \hf \gamma^2\ps)\;,
\end{equ}
in order to shorten  the expressions.

As before, we see that there exists a positive constant $\bar C$ and a 
function $\bar \CR\colon \R^2 \to \R_+$ with compact support such that $\Phi^2 - \bar C \Hf^{{2\over k} - 1} + \bar \CR$
is smooth, centred, and vanishes in a neighbourhood of the origin. Similarly to \eref{e:defXi}, we define $\Xi$ as the unique centred solution to
\begin{equ}
\CL_0 \Xi(P,Q) = \Phi^2(P,Q) - \bar C \Hf^{{2\over k} - 1}(P,Q) + \bar \CR(P,Q)\;,
\end{equ}
and $\Psi$ as the unique centred solution to $\CL_0 \Psi = \Phi$. Note that $\Xi$ scales like $\Hf^{{5\over 2k} - {3\over 2}}$ and that
$\Psi$ scales like $\Hf^{{3\over 2k} - 1}$. 

At this stage, we would like to define
$V = \hat V + \alpha_\ps \Xi(p_1, q_1) - c_\ps \tilde p_0 \Psi(p_1,q_1)$ in order to compensate for the last two terms in \eref{e:boundVhat}. 
The problem is that when applying the generator to
$\tilde p_0 \Psi$, we obtain an unwanted term of the type $q_0|q_0|^{2k-2}\Psi$, which grows too fast in the $q_0$ direction.
We note however that the term $\tilde p_0 \Phi$ only needs to be compensated when $|\tilde p_0| \gg \Phi$, which is the regime
in which the description \eref{e:dyntilde} is expected to be relevant. We therefore consider the same cutoff function $\psi$ as before
and we set
\begin{equ}[e:deffinalV]
V = \hat V + \alpha_\ps \Xi(p_1, q_1) - c_\ps \tilde p_0 \Psi(p_1,q_1) \psi \Bigl({1+\Hf(\tilde p_0,q_0) \over (1 + \Hf(p_1, q_1))^\cut}\Bigr)\;,
\end{equ}
for a positive exponent $\cut$ to be determined later.

In order to obtain bounds on $\CL V$, we make use of the fact that Lemma~\ref{lem:scalingLf2}
still applies to the present situation. In particular, we can apply it to the function $\Xi$, thus obtaining the bound
\begin{equ}
\CL V \lesssim  -  C_\ps\bigl(\Hf(\tilde p_0,q_0) + \Hf^{{2\over k}-1}(p_1, q_1)\bigr) +  c_\ps \Bigl(\tilde p_0 \Phi
- \CL\bigl( \tilde p_0 \Psi(p_1,q_1) \psi \bigr)\Bigr)\;,
\end{equ}
for some constant $C_\ps$, where it is understood that the function $\psi$ is composed with
the ratio appearing in \eref{e:deffinalV}. Using the fact that $\CL_0 \Psi = \Phi$ by definition and applying
the chain rule \eref{e:chainrule} for $\CL$, we thus obtain
\begin{equs}
\CL V &\lesssim - C_\ps\bigl(\Hf(\tilde p_0,q_0) + \Hf^{{2\over k}-1}(p_1, q_1)\bigr)\\
&\quad -  c_\ps \bigl( (\CL\tilde p_0) \Psi \psi + \tilde p_0 \Psi \CL \psi + \tilde p_0 (\CL - \CL_0)\Psi + \tilde p_0 \CL \Psi (\psi-1))\bigr) \label{e:bestboundV}\\
&\quad  -  c_\ps T\bigl(\d_{p_1} \tilde p_0 \d_P\Psi \psi + \d_{p_1}\tilde p_0 \Psi \d_{p_1}\psi  + \tilde p_0 \d_P\Psi \d_{p_1}\psi\bigr) -c_\ps T_\infty \Psi \d_{p_0}\psi\;.
\end{equs}
We claim that all the terms appearing on the second and the third line of this expression are negligible, thus concluding the proof. 
The most tricky part of showing this is to obtain bounds on $\CL \psi$.

Define $E_0 = 1+\Hf(\tilde p_0,q_0)$ and $E_1 = 1+\Hf(p_1,q_1)$ as a shorthand.
Our main tool in bounding $\CL V$ is then the following result which shows that the terms containing $\CL \psi$ are negligible: 
\begin{proposition}\label{prop:boundsLpsi}
Provided that $\cut \in [2-k,k]$, there exists a constant $C$ such that
\begin{equ}
\Bigl|\CL \psi \Bigl({E_0 \over E_1^\cut}\Bigr)\Bigr| \le C\;,\quad
\Bigl|\d_{p_0}\psi \Bigl({E_0 \over E_1^\cut}\Bigr)\Bigr| \le C E_1^{-{\cut \over 2}}\;,\quad 
\Bigl|\d_{p_1} \psi \Bigl({E_0 \over E_1^\cut}\Bigr)\Bigr| \le CE_1^{-{1 \over 2}}\;.
\end{equ}
\end{proposition}
\begin{proof}
Define the function $f\colon \R_+^2 \to \R_+$ by $f(x,y) = \psi((1+x)/(1+y)^\cut)$. It can then be checked by induction that, for every
pair of positive integers $m$ and $n$ with $m+n>0$ and for every real number $\beta$, there exists a constant $C$ such that the bound
\begin{equ}[e:boundderf]
|\d_x^m \d_y^n f| \le (1+x)^{-m+ \beta} (1+y)^{-n - \cut\beta}\;.
\end{equ}
holds uniformly in $x$ and $y$. It furthermore follows from \eref{e:dyntilde} and \eref{e:model} that 
\begin{equs}[2]
\multicol{3}{|\CL E_0| \le C\bigl(E_0+E_0^{1- {1\over 2k}}E_1^{{1\over k} - {1\over 2}}\bigr)\;,} \\
\multicol{3}{|\CL E_1| \le C\bigl(E_1^{{1\over 2} + {1\over 2k}}+E_0^{{1\over 2k}}E_1^{{1\over 2}}\bigr)\;,}\\
|\d_{p_0} E_0| &\le C E_0^{1/2}\;, &\qquad |\d_{p_0} E_1| &= 0\;,\\
|\d_{p_1} E_0| &\le C E_0^{{1\over 2}}E_1^{{1\over k} - {1}}\;, &\qquad |\d_{p_1} E_1| &\le C E_1^{1\over 2}\;.
\end{equs}
Combining these two bounds with \eref{e:boundderf} and the
chain rule \eref{e:chainrule}, the required bounds follow.
\end{proof}

Let us now return to the bound on $\CL V$. It is straightforward to check that
\begin{equ}
|\CL \tilde p_0| \le C \bigl(E_0^{1-{1\over 2k}} + E_1^{{1\over k} - {1\over 2}}\bigr)\;,
\end{equ}
for some constant $C$, so that
\begin{equ}
|\Psi(p_1,q_1)\CL \tilde p_0| \le C \bigl(E_0^{1-{1\over 2k}}E_1^{{3\over 2k} - 1} + E_1^{{5\over 2k} - {3\over 2}}\bigr)\;,
\end{equ}
which is negligible. Combining Proposition~\ref{prop:boundsLpsi} with the scaling behaviours of $\Phi$ and $\Psi$,
one can check in a similar way that the term $\tilde p_0 \Psi \CL \psi$, as well as all the terms appearing
on the third line of \eref{e:bestboundV} are also negligible. It therefore remains to bound
$\tilde p_0 (\CL - \CL_0)\Psi$ and $\tilde p_0 \CL \Psi (\psi-1)$.
It follows from \eref{e:exprg} that
\begin{equ}[e:boundLPsi]
|(\CL - \CL_0)\Psi| \le C E_1^{{3\over 2k}-{3\over 2}}\bigl(E_0^{1\over 2k} + E_1^{1\over 2k}\bigr)
\le C \bigl(E_0^{1\over 2k}+E_1^{{2\over k} - {3\over 2}}\bigr)\;,
\end{equ}
so that $|\tilde p_0 (\CL - \CL_0)\Psi|$ is negligible as well. Since we know that $\CL_0 \Psi$ scales like $E_1^{{1\over k}-{1\over 2}}$,
it follows from \eref{e:boundLPsi} that
\begin{equ}
|\tilde p_0 \CL \Psi | \le C E_0^{1\over 2}E_1^{{3\over 2k}-{3\over 2}}\bigl(E_0^{1\over 2k} + E_1^{1-{1\over 2k}}\bigr)\;.
\end{equ}
This term has of course no chance of being negligible: we have to use the fact that it is multiplied by
$1-\psi$. The function $1-\psi$ is non-vanishing only when $E_0 \ge E_1^\cut$, so that we obtain
\begin{equ}
|\tilde p_0 \CL \Psi(1-\psi) | \le C \bigl(E_0^{{1\over 2k}+{1\over2} + {1\over \cut} ({3\over 2 k}-{3\over 2})}
+ E_0^{{1\over 2} + {1\over \cut}({1\over k} - {1\over 2})}\bigr)\;.
\end{equ}
We see that both exponents are strictly smaller than $1$, provided that $\cut > {2\over k} - 1$. Combining
all of these estimates with \eref{e:bestboundV}, we see that, provided that $\cut \in ({2\over k} - 1,k)$, there exists a constant $C$
such that
\begin{equ}
\CL V \lesssim - C \bigl(E_0 + E_1^{{2\over k}-1}\bigr)\;.
\end{equ}
In particular, using the scaling of $\Phi$, we deduce the existence of a constant $c$ such that the bound
\begin{equ}[e:finalboundV]
\CL V \le  - c V^{{2\over k} - 1}\;,
\end{equ}
holds outside of a sufficiently large compact set (we can choose such a set so that $V$ is positive outside), 
thus concluding the proof of Theorem~\ref{theo:positive} by
applying Theorem~\ref{theo:positiveAbstract}.
\end{proof}

\section{Integrability properties of the invariant measure}
\label{sec:Integral}

The aim of this section is to explore the integrability properties of the invariant measure $\mu_\star$ when it exists.
First of all, we show the completely unsurprising  fact that:

\begin{proposition}\label{prop:nonintExpH}
For all ranges of parameters for which there exists an invariant measure $\mu_\star$, 
one has $\int \exp \bigl(\beta H(x)\bigr)\,\mu_\star(dx) = \infty$ for every $\beta > 1/T$.
\end{proposition}

\begin{proof}
Choose $\beta > \beta_2 > 1/T$. Setting $W_2(x) = \exp(\beta_2 H(x))$, we
have
\begin{equ}
\CL W_2 = \gamma \beta_2W_2 \Bigl(T + T_\infty -p_0^2 + \beta_2 \bigl(T p_0^2 + T_\infty p_1^2\bigr)\Bigr)
\le \exp(\beta H)\;,
\end{equ}
outside of a sufficiently large compact set.
Setting similarly $W_1 =  \exp(H(x)/T)$, we see immediately from a similar calculation that 
$\CL W_1 \ge 0$, so that the result follows from Theorem~\ref{theo:negativeAbstract}.
\end{proof}

\begin{remark}
Actually, one can show similarly a slightly stronger result, namely that there exists some exponent $\alpha < 1$ such that
$H^\alpha \exp(H/T)$ is not integrable against $\mu_\star$. 
\end{remark}

\subsection{Energy of the first oscillator}

What is maybe slightly more surprising is that the tail behaviour of the distribution of the energy of the first oscillator
is not very strongly influenced by the presence of an infinite-temperature heat bath just next to it, provided that
we look at the correct set of variables. Indeed, we have:

\begin{proposition}\label{prop:tailsH0}
Let either ${3\over 2} \le k < 2$ or $k = 2$ and $T_\infty$ be such that there exists an invariant probability measure $\mu_\star$.
Then $\int \exp \bigl(\beta \Hf(\tilde p_0, q_0)\bigr)\,\mu_\star(dx) < \infty$ for every $\beta < 1/T$.
\end{proposition}

\begin{remark}
When $k=2$, $\Phi$ is bounded and the exponential integrability of $\Hf(\tilde p_0,q_0)$ is equivalent 
to that of $\Hf(p_0,q_0)$. This is however \textit{not} the case when $k < 2$.
\end{remark}

\begin{remark}
The borderline case $k = {3\over 2}$ is expected to be optimal if we restrict ourselves
to the variables $(\tilde p_0, q_0)$. This is because for $k < {3\over 2}$ one would have to add 
additional correction terms taking into account the nonlinearity of the pinning potential.
\end{remark}

The main ingredient in the proof of Proposition~\ref{prop:tailsH0} is the following proposition, which is also going to
be very useful for the non-integrability results later in this section.

\begin{proposition}\label{prop:propertiesH0}
For every $\ps>0$, there exist functions $\HO, \hat p_0 \colon \R^4 \to \R$ and a constant $C_\ps$ such that
\begin{claim}
\item For every $\eps > 0$, there exists a constant $C_\eps$ such that the bounds
\begin{equ}[e:upperH0]
0 \le \HO \le (1+\eps) H + C_\eps\;,
\end{equ}
hold.
\item Provided that $k \ge {3\over 2}$, for every $\eps>0$ there exists a constant $C_\eps$ such that the bound
\begin{equ}[e:twosidedH0]
(1-\eps) \Hf(\tilde p_0,q_0) - C_\eps \le \HO \le (1+\eps) \Hf(\tilde p_0, q_0) + C_\eps\;,
\end{equ}
holds.
\item One has the bounds 
\minilab{e:bounddpH0}
\begin{equs}
(\d_{p_0}\HO - \hat p_0)^2 &\le C_\ps + \ps^4 \HO\;,\label{e:bounddp0H0}\\
(\d_{p_1}\HO)^2 &\le C_\ps + \ps^4 \HO\;.\label{e:bounddp1H0}
\end{equs} 
\item If furthermore $k \ge {3\over 2}$, the bound $\CL \HO \le C_\ps - (\gamma - 2\ps) \hat p_0^2 - \ps \HO$ holds.
\item If $k \in (4/3, 3/2)$ then, for every $\delta > (2k-1)({3\over k}-2)$, one has the bound $\CL \HO \le C_\ps - (\gamma - 2\ps) \hat p_0^2 - \ps \HO + \ps^2\Hf^\delta (p_1,q_1)$.
\end{claim}
\end{proposition}

\begin{remark}
The presence of $\tilde p_0$ rather than $\hat p_0$ in \eref{e:twosidedH0} is not a typographical mistake.
\end{remark}

\begin{proof}
We start be defining the differential operator $\CK$ acting on functions $F\colon \R^2 \to \R$ as
\begin{equ}
\CK F = \gamma T_\infty \bigl(\d_P^2 F\bigr)(p_1,q_1) + \bigl(\alpha\bigl(\hat q_0 - \Phi_q(p_1,q_1) - q_1\bigr) - R_1'(q_1)\bigr)\bigl(\d_P F\bigr)(p_1,q_1) \;,
\end{equ}
so that $\CK F = \CL \bigl(F(p_1,q_1)\bigr) - \bigl(\CL_0 F\bigr)(p_1,q_1)$.
Setting
\begin{equ}[e:pqhat]
\hat p_0 = p_0 + \Phi_p(p_1,q_1)\;,\qquad \hat q_0 = q_0 + \Phi_q(p_1, q_1)\psi(E_0/E_1^\cut) \;,
\end{equ}
for some yet to be defined functions $\Phi_p$ and $\Phi_q$ and for $E_i$ and $\psi$ as in Proposition~\ref{prop:boundsLpsi}, we then obtain
\begin{equs}[e:dynamichat1]
d\hat q_0 &= \hat p_0\,dt + \bigl(\CL_0 \Phi_q - \Phi_p\bigr)\,dt   \\
&\quad + \psi \CK\Phi_q\,dt + (\psi-1) \CL_0 \Phi_q\,dt + \Phi_q \CL \psi\,dt  + \gamma T_\infty \d_{p_1}\psi \d_P \Phi_q\,dt\\
&\quad + \sqrt{2\gamma T_\infty} \bigl(\psi \d_P \Phi_q + \Phi_q \d_{p_1}\psi \bigr)\,dw_1(t) + \sqrt{2\gamma T} \Phi_q \d_{p_0}\psi \,dw_0(t)\;,\\
d\hat p_0 &= - \VV'(\hat q_0)\,dt - \gamma \hat p_0\,dt + \sqrt{2\gamma T}dw_0 \\
&\quad + \bigl(\CL_0 \Phi_p - \alpha q_1 + \gamma \Phi_p\bigr)\,dt \\
&\quad + \CK \Phi_p\,dt + \bigl(\VV'(\hat q_0) - \VV'(q_0)\bigr)\,dt + \sqrt{2\gamma T_\infty} \d_P \Phi_p\,dw_1(t)\;,
\end{equs}
where we defined as before the effective potential $\VV(q) = {V_1(q)} + \alpha {q^2 \over 2}$.

Let $E>0$ and set $\Phi_p^{(1)}$ as the unique centred solution to
\begin{equ}[e:defPhip1]
\CL_0 \Phi_p^{(1)} = \alpha Q \bigl(1-\psi\bigl(\Hf(P,Q)/E\bigr)\bigr)\;,
\end{equ}
where $\psi$ is the same cutoff function already used previously. We then define $\Phi_p^{(2)}$
by $\CL_0 \Phi_p^{(2)} = \gamma \Phi_p^{(1)}$ and we set $\Phi_p = \Phi_p^{(1)} + \Phi_p^{(2)}$. This ensures
that one has the identity
\begin{equ}
\CL_0 \Phi_p - \alpha q_1 + \gamma \Phi_p = \CR_p\;,
\end{equ}
where the function $\CR_p$ consists of terms that scale like $\Hf^a$ with $a \le {3\over2k}-1$.
We furthermore set $\Phi_q$ to be the unique centred solution to $\CL_0 \Phi_q = \Phi_p$.
Note that $\Phi_p$ consists of terms scaling like $\Hf^a$ with $a \le {1\over k} - {1\over 2}$ and 
that $\Phi_q$ consists of terms scaling like $\Hf^a$ with $a \le {3\over 2k} - 1$. 
The introduction of the parameter $E$ in \eref{e:defPhip1} ensures that we can make functions
scaling like a negative power of $\Hf$ arbitrarily small in the supremum norm. It follows indeed that
one has for example $|\d_P \Phi_p| \le C E^{{1\over k}-1}$.

With these definitions at hand,  it follows from \eref{e:dynamichat1} that
\begin{equs}[e:dynamichat]
d\hat q_0 &= \hat p_0\,dt + \sqrt{2\gamma T_\infty} \bigl(\psi \d_P \Phi_q + \Phi_q \d_{p_1}\psi \bigr)\,dw_1(t) + \sqrt{2\gamma T} \Phi_q \d_{p_0}\psi \,dw_0(t) \\
&\quad + \psi \CK\Phi_q\,dt + (\psi-1) \Phi_p\,dt + \Phi_q \CL \psi\,dt  + \gamma T_\infty \d_{p_1}\psi \d_P \Phi_q\,dt\;,\\
d\hat p_0 &= - \VV'(\hat q_0)\,dt - \gamma \hat p_0\,dt + \sqrt{2\gamma T}dw_0 \\
&\quad + \CR_p\,dt + \CK \Phi_p\,dt + \bigl(\VV'(\hat q_0) - \VV'(q_0)\bigr)\,dt + \sqrt{2\gamma T_\infty} \d_P \Phi_p\,dw_1(t)\;.
\end{equs}
Let now $\HO$ be defined by
\begin{equ}
\HO = {\hat p_0^2 \over 2} + \VV(\hat q_0) + \ps \hat p_0 \hat q_0 + C_0\;,
\end{equ}
were $C_0$ is a sufficiently large constant so that $\HO \ge 1$.
Note that, as a consequence of the definitions of $\hat p_0$ and $\hat q_0$, if $k \ge 3/2$ then $|\hat p_0 - \tilde p_0|$ and $|\hat q_0 - q_0|$
are bounded so that the two-sided bound \eref{e:twosidedH0} does indeed hold. Showing that the weaker one-sided bound
\eref{e:upperH0} holds for every $k \in [{3\over 2},2]$ is straightforward to check.

Before we turn to the proof of \eref{e:bounddpH0}, let us 
define $\hat E_0$ in a similar way as in the proof of the case $k < 2$ of Theorem~\ref{theo:positive},
but using the `hat' variables. If $k \ge {3\over 2}$, then $\hat E_0$ and $E_0$ are equivalent in the sense that they are bounded by
multiples of each other. If $k < {3\over 2}$, this is not the case, but it follows from the definitions of $\Phi_p$ and $\Phi_q$ that
\begin{equ}
\hat E_0 \le C \bigl(E_0 + E_1^{3-2k}\bigr)\;,\qquad
E_0 \le C \bigl(\hat E_0 + E_1^{3-2k}\bigr)\;.
\end{equ}
It follows that, provided that we impose the condition $\cut > 3-2k$, where $\cut$ is the exponent appearing in \eref{e:pqhat}, 
then one has the implications
\minilab{e:cutoff}
\begin{equs}
E_0 &\le C E_1^\cut \quad \Rightarrow\quad \hat E_0 \le \tilde C E_1^\cut \label{e:cutoff1}\\
E_0 &\ge C E_1^\cut \quad \Rightarrow\quad \hat E_0 \ge \tilde C E_1^\cut\;,\label{e:cutoff2}
\end{equs}
for some constant $\tilde C$ depending on $C$. We will assume from now on that the condition $\cut > 3-2k$ is indeed satisfied.
Let us now show that \eref{e:bounddp0H0} holds. We have the identity
\begin{equ}
\d_{p_0} \HO - \hat p_0 = \ps \hat q_0 + \bigl(\VV'(\hat q_0) + \ps \hat p_0\bigr)\Phi_q \d_{p_0} \psi\;.
\end{equ}
Since the term $\ps \hat q_0$ satisfies the required bound, we only need to worry about the second term. 
It follows from Proposition~\ref{prop:boundsLpsi} and from the scaling of $\Phi_q$ that
this term is bounded by a multiple of $\hat E_0^{1 - {1\over 2k}} E_1^{{3\over 2k} - 1 - {\cut \over 2}}$.
Since the bounds \eref{e:cutoff} hold on the support of $\d_{p_0} \psi$, this in turn is bounded by
a multiple of $\hat E_0^{1/2} E_1^{{3\over 2k} - 1 - {\cut \over 2k}}$,
so that the requested bound follows, provided again that the condition $\cut > 3-2k$ holds.  

Turning to \eref{e:bounddp1H0}, we have the identity
\begin{equ}
\d_{p_1} \HO = (\hat p_0 + \ps \hat q_0) \d_P \Phi_p + \bigl(\VV'(\hat q_0) + \ps \hat p_0\bigr)\d_{p_1}(\Phi_q \psi)\;.
\end{equ}
Making use of the parameter $E$ introduced in \eref{e:defPhip1}, it follows that the first term is bounded by
$\hat E_0^{1/2} E^{{1\over k}-1}$, which can be made sufficiently small by choosing $E \gg \ps ^{2k\over 1-k}$.
In order to bound the second term, we expand the last factor into $\Phi_q \d_{p_1}\psi + \psi \d_{P}\Phi_q$.
The first term can be bounded just as we did for $\d_{p_0}\HO$, noting that the bound on $\d_{p_1}\psi$ in Proposition~\ref{prop:boundsLpsi}
is better than the bound on $\d_{p_0}\psi$. Using the fact that \eref{e:cutoff1} holds on the support of $\psi$,
the second term yields a bound of the form $\hat E_0^{1\over 2} E_1^{{\cut - 3 \over 2}(1-{1\over k})}$, which yields the
required bound provided that $\cut < 3$.

It therefore remains to show the bound on $\CL \HO$.
It follows from \eref{e:dynamichat} that one has the identity
\begin{equs}[e:LHO]
\CL \HO & = \gamma T - (\gamma - \ps) \hat p_0^2  - \ps |\hat q_0|^{2k} - \alpha \ps |\hat q_0|^2 - \gamma \ps \hat p_0 \hat q_0 \\
&\quad + \gamma T_\infty \bigl((\d_P \Phi_p)^2 + \VV''(\hat q_0) \bigl(\d_{p_1}(\Phi_q \psi) \bigr)^2 + 2\ps \d_P \Phi_p\d_{p_1}(\Phi_q \psi)\bigr) \\
&\quad + \gamma T \bigl(\VV''(\hat q_0)(\Phi_q \d_{p_0}\psi)^2 + 2\ps \Phi_q \d_{p_0}\psi \bigr) \\
&\quad + (\hat p_0 + \ps \hat q_0)\bigl(\CR_p + \CK \Phi_p + \VV'(\hat q_0) - \VV'(q_0)\bigr) \\
&\quad + \bigl(\VV'(\hat q_0) + \ps \hat p_0\bigr) \bigl(\psi \CK\Phi_q + (\psi-1) \Phi_p + \Phi_q \CL \psi  + \gamma T_\infty \d_{p_1}\psi \d_P \Phi_q\bigr)\;.
\end{equs}
We now use the following notion of a negligible function. A function $f \colon \R_+\times \R^4 \to \R$ is negligible if, for every $\eps>0$
there exists a constant $E_\eps$ and, for every $E > E_\eps$, there exists a constant $C_\eps$ such that
the bound $|f(E;p,q)| \le C_\eps + \eps \bigl(\hat E_0 + E_1^\delta\bigr)$ holds, where $\delta$ is as in the statement of the proposition.
(Set $\delta = 0$ for $k \ge 3/2$.)

With this notation, the required bounds follow if we can show that all the terms appearing in \eref{e:LHO} 
are negligible, except for those on the first line. The terms appearing in the second line are all smaller than the last
term appearing in $\d_{p_1}\HO$ and so they are negligible. Similarly, the terms appearing in the third line are
smaller than those appearing in $\d_{p_1}\HO - \hat p_0$.

It is easy to see that the first term on the fourth line is negligible. Concerning the second term, we see that
$|\CK \Phi_p| \le C \bigl(\hat E_0^{1\over 2k} + E_1^{{3\over 2k}-1}\bigr)$, so that this term is also seen to be negligible
by power counting.
Note now that the definitions of $\VV$ and $\hat q_0$ imply that one has the bound
\begin{equs}
\bigl|\VV'(\hat q_0) - \VV'(q_0)\bigr| &\le C \bigl(1 + |\hat q_0|^{2k-2} + |q_0|^{2k-2}\bigr) \Hf^{{3\over 2k}-1}(p_1,q_1) \\
&\le C \bigl(1 + |\hat q_0|^{2k-2} + E_1^{(2k-2)({3\over 2k}-1)}\bigr) E_1^{{3\over 2k}-1} \\
& \le C \bigl(\hat E_0^{1 - {1\over k}} E_1^{{3\over 2k}-1} + E_1^{(2k-1)({3\over 2k}-1)}\bigr)\;.
\end{equs}
Furthermore, one has $\VV'(\hat q_0) = \VV'(q_0)$, unless $\hat E_0 \le E_1^\cut$, so that we have the bound
\begin{equ}
\hat p_0 \bigl|\VV'(\hat q_0) - \VV'(q_0)\bigr| \le C \bigl(\hat E_0 E_1^{\cut({1\over 2} - {1\over k}) + {3\over 2k}-1}
	+ \hat E_0^{1\over 2} E_1^{(2k-1)({3\over 2k}-1)}\bigr)\;.
\end{equ}
The second term is always negligible. Furthermore, if $\cut > (3-2k)/(2-k)$ the first term is also negligible.

We now turn to the last line in \eref{e:LHO}. In order to bound the term involving $\CK \Phi_q$, note that 
the functions $\Phi_q \d_P \Phi_q$, $\d_P^2\Phi_q$, and $Q\d_P \phi_q$ are bounded provided that $k \ge {4\over 3}$, so that 
the terms involving these expressions are negligible. Concerning the term $\VV'(\hat q_0)\hat q_0  \d_P \Phi_q$, we
use the fact that $\d_P \Phi_q$ can be made arbitrarily small by choosing $E$ large enough in \eref{e:defPhip1} to conclude
that it is also negligible. The term involving $\Phi_p$ is bounded by a multiple of $\hat E_0^{1-{1\over 2k} + {1\over \cut}({1\over k}-{1\over 2})}$,
so that it is negligible provided that $\cut > 2-k$. The term involving $\Phi_q\CL\psi$ is bounded similarly, using the fact that
$\CL\psi$ is bounded by Proposition~\ref{prop:boundsLpsi} and that $\Phi_q$ scales like a smaller power of $\Hf$ than $\Phi_p$.
Finally, the last term is negligible since $\d_{p_1}\psi \d_P\Phi_q$ is bounded, thus concluding the proof of Proposition~\ref{prop:propertiesH0}. Note that the choice $\cut = 2$ for example allows to satisfy all the conditions that we had to impose on $\cut$
in the interval $k \in [4/3,2]$.
\end{proof}

We are now able to give the

\begin{proof}[of Proposition~\ref{prop:tailsH0}]
It follows from \eref{e:twosidedH0} that if we can show that $\exp(\beta \HO)$ is integrable
with respect to $\mu_\star$ for every $\beta < 1/T$, then the same is also true for $\exp(\beta \Hf(\tilde p_0, q_0))$,
provided that we restrict ourselves to the range $k \ge {3\over 2}$.

Before we proceed, we also note that \eref{e:bounddp0H0} implies that for $\ps$ sufficiently small,
one has the bound
\begin{equs}
(\d_{p_0}\HO)^2 \le (1+\ps)\hat p_0^2 + \tilde C_\ps + \ps^2 \HO\;,
\end{equs}
for some constant $\tilde C_\ps$.
Setting $W = \exp(\beta \HO)$, we thus have the bound
\begin{equs}
{\CL W\over \beta W} &= \CL \HO + \gamma \beta \bigl(T(\d_{p_0}\HO)^2 + T_\infty(\d_{p_1}\HO)^2\bigr) \\
&\le C_\ps - (\gamma - 2\ps)\hat p_0^2 - \ps \HO + \gamma \beta (1+\ps)\hat p_0^2 + C \ps^2 \HO\;,
\end{equs}
for some constant $C$ independent of $\ps$. Since we assumed that $\beta < 1/T$, we can make $\ps$ sufficiently
small so that $-(\gamma - 2\ps) + \gamma \beta (1+\ps) < 0$ and $C\ps^2 - \ps < 0$. The claim then
follows from Theorem~\ref{theo:positiveAbstract}.
\end{proof}

\subsection{Integrability and non-integrability in the case $\pmb{k=2}$}

We next show that if $k = 2$ and $T_\infty \le \alpha^2 \hat C$, then the invariant measure
is heavy-tailed in the sense that there exists an exponent $\zeta$ such that $\int H^\zeta(x)\,\mu_\star(dx) = \infty$.
Our precise result is given by:

\begin{theorem}\label{theo:tail}
If $k=2$ and $T_\infty \le \alpha^2 \hat C$, one has $\int H^\zeta(x)\,\mu_\star(dx) = \infty$ provided that
\begin{equ}
\zeta > \zeta_\star \eqdef {3\over 4} {\alpha^2 \hat C -T_\infty \over T_\infty}\;.
\end{equ}
Conversely, one has $\int H^\zeta(x)\,\mu_\star(dx) < \infty$ for $\zeta < \zeta_\star$.
\end{theorem}

\begin{proof}
We first show the positive result, namely that $H^\zeta$ is integrable with respect to $\mu_\star$ for any $\zeta < \zeta_\star$.
Fixing such a $\zeta$, our aim is to construct a smooth function $W$ bounded from below such that, for some small value $\eps>0$, the bound
 $\CL W \le -\eps H^\zeta$ holds outside of some compact set. This then immediately implies the required integrability
 by Theorem~\ref{theo:positiveAbstract}.

Consider the function $V$ defined in \eref{e:defV2}. Note that this function depends on parameters
$E$, $\ps$ and $c$ and that, for any given value of $\eps>0$,  it is possible to choose first  $\ps$ sufficiently small
 and $c$ sufficiently close to $1$, and then $E$
sufficiently large, so that the bound
\begin{equ}
\CL V \le \gamma T_\infty - \alpha^2\gamma\hat C + \eps \;, 
\end{equ}
holds outside of some compact set. 

Let us now turn to the behaviour of $\d_{p_0}V$ and $\d_{p_1}V$. It follows from the definitions, Lemma~\ref{lem:negligible},
 and Corollary~\ref{cor:neglderpsiE} that one has the identity
\begin{equ}
\bigl(\d_{p_0} V\bigr)^2 = (1-c)^2 p_0^2 + R_0\;,
\end{equ}
where the function $R_0$ can be bounded by an arbitrarily small multiple of $V$ outside of some sufficiently large compact set. 
Furthermore, it follows from the definition of $V$ and the
construction of $H_0$ that one has the bound $V \ge {1-c \over 2} H$ outside of some compact set, so that we have the bound
\begin{equ}
\bigl(\d_{p_0} V\bigr)^2 \le 4(1-c)V + R_0\;.
\end{equ}
Ensuring first that $1-c\le \eps/8$ and then choosing $E$ sufficiently large, it follows that we can ensure that 
$\bigl(\d_{p_0} V\bigr)^2 \le \eps V$ outside of a sufficiently large compact set. It follows in a similar way that,
by possibly choosing $E$ even larger, the bound
\begin{equ}
\bigl(\d_{p_1} V\bigr)^2 \le p_1^2 + \eps V
\end{equ}
holds outside of some compact set. Note now that since 
\begin{equ}[e:averagep2]
\CL_0 (PQ) = 3P^2 - 4\Hf\;,
\end{equ}
the function $P^2 - {4\over 3} \Hf$ is centred.
Let furthermore $\tilde \CR\colon \R^2 \to \R$ be a centred compactly supported function
such that $P^2 - {4\over 3} \Hf + \tilde \CR$ vanishes in a neighbourhood of the origin and let $\tilde \Xi$ be the centred
solution to
\begin{equ}
\CL_0 \tilde \Xi = P^2 - {4\over 3} \Hf + \tilde \CR\;,
\end{equ}
so that we have the identity
\begin{equ}
\CL \tilde \Xi(p_1,q_1) = p_1^2 - {4\over 3} \Hf(p_1,q_1) + \tilde \CR(p_1,q_1) +\bigl(\alpha (q_0-q_1) - R_1'(q_1)\bigr)\bigl(\d_P\tilde \Xi\bigr)(p_1,q_1)\;.
\end{equ}
Furthermore, it follows at once from the definition of $V$ and the scaling behaviours of $\Xi$ and $\Psi$ that the bound
\begin{equ}
\Hf(p_1,q_1) \le (1+\eps) V\;,
\end{equ}
holds outside of some compact set. Since furthermore $\tilde \CR$ is bounded and $\Xi$ scales like $\Hf^{3\over 4}$, it follows
that the bound
\begin{equ}
\CL \tilde \Xi(p_1,q_1) \ge p_1^2 - {4\over 3}(1+\eps) V\;,
\end{equ}
holds outside of some (possibly larger) compact set. Finally, it follows from the scaling of $\tilde \Xi$ that
the bounds
\begin{equ}[e:boundrest]
|\tilde \Xi \CL V| \le \eps V \quad \hbox{and}\quad |\d_{p_1}V\d_{p_1}\tilde \Xi| \le \eps V\;,
\end{equ}
hold outside of some sufficiently large  compact set.

With all these definitions at hand, we consider the function
\begin{equ}[e:defW]
W = V^{\zeta + 1} - \gamma \zeta(\zeta+1) T_\infty V^\zeta \tilde \Xi(p_1,q_1)\;.
\end{equ}
Note that $V$ is positive outside of a compact set, so that $W$ is well-defined there.
Since we do not care about compactly supported modifications of $W$, we can assume that \eref{e:defW} makes sense
globally.
We then have the identity
\begin{equs}
\CL W &= (\zeta+1) V^\zeta \CL V + \zeta\gamma (\zeta+1) V^{\zeta-1} \bigl(T\bigl(\d_{p_0} V\bigr)^2 +T_\infty\bigl(\d_{p_1} V\bigr)^2 -T_\infty\CL \tilde \Xi\bigr) \\
&\quad - \gamma \zeta^2 (\zeta+1)T_\infty V^{\zeta - 1} \bigl(\tilde \Xi \CL V + \gamma T_\infty \d_{p_1}V\d_{p_1}\tilde \Xi  \bigr)\;. 
\end{equs}
Collecting all of the bounds obtained above, this in turn yields the bound
\begin{equs}
\CL W &\le (\zeta+1) V^\zeta \bigl(\gamma T_\infty - \alpha^2 \gamma \hat C + \eps\bigr) + \zeta\gamma (\zeta+1) V^{\zeta} \bigl(T\eps 
+ T_\infty \eps + {4\over 3} T_\infty (1+\eps)\bigr)  \\
&\quad - \gamma \eps \zeta^2 (\zeta+1)T_\infty V^\zeta \\
&\le \gamma (\zeta+1) \bigl(T_\infty - \alpha^2 \hat C + {4\over 3}\zeta T_\infty  + K\eps \bigr)V^\zeta \;,
\end{equs}
holding for some constant $K>0$ independent of $\eps$ outside of some sufficiently large compact set. It follows that 
if $\zeta  < \zeta_\star$, it is possible to choose $\eps$ sufficiently small so that the prefactor in this expression is negative,
thus yielding the desired result.

We now prove the `negative result', namely that $H^\zeta$ is \textit{not} integrable with respect to $\mu_\star$ if
$\zeta > \zeta_\star$. In order to show this, we are going to apply Wonham's criterion with $W_2= H^{1+\zeta}$.
It therefore suffices to find a function $W_1$ growing to infinity in some direction, such that $\CL W_1 > 0$ outside
of some compact set, and such that 
\begin{equ}[e:condW]
\sup_{H(p,q) = E}{W_1(p,q) E^{-1-\zeta}} \to 0
\end{equ}
as $E \to \infty$. We are going to
construct $W_1$ in a way very similar to the construction in the proof of the positive result above. 

Fix some arbitrarily small $\eps > 0$ as before. Setting $V$ as above, note first that it follows immediately from \eref{e:boundH0} 
that, by choosing first $\ps$ sufficiently small, then $c$ sufficiently close to $1$ and finally $E$ large enough, we can ensure
that the bound
\begin{equ}
\CL V \ge \gamma T_\infty - \alpha^2 \gamma \hat C - \eps (1 + p_0^2)
\end{equ}
holds outside of some sufficiently large compact set. Similarly as before, we can also ensure that the bound
\begin{equ}
\bigl(\d_{p_1} V\bigr)^2 \le p_1^2 - \eps V
\end{equ}
holds. Fix now some $\tilde \zeta \in (\zeta_\star, \zeta)$ and define $W_0$ as in \eref{e:defW}, but with 
$\tilde \zeta$ replacing $\zeta$.  It follows that the bound
\begin{equ}
\CL W_0 \ge \gamma (\tilde \zeta+1) \bigl(T_\infty - \alpha^2 \hat C + {4\over 3}\tilde\zeta T_\infty  - K\eps(1+p_0^2) \bigr)V^{\tilde \zeta} \;,
\end{equ}
holds for some constant $K>0$ outside of some compact set. The problem is that the right hand side of this expression is not
everywhere positive because of the appearance of the term $p_0^2$. This can however be dealt with by setting
\begin{equ}[e:defW2]
W_1 = W_0 - K \eps H^{1+\tilde \zeta}\;,
\end{equ}
so that
\begin{equ}
\CL W_1 \ge \gamma (\tilde \zeta+1) \bigl(T_\infty - \alpha^2 \hat C + {4\over 3}\tilde\zeta T_\infty  - \tilde K\eps\bigr)V^{\tilde \zeta} \;,
\end{equ}
for some different constant $\tilde K$. Since $\tilde \zeta > \zeta_\star$, we can ensure that this term is uniformly positive by 
choosing $\eps$ sufficiently small. By possibly making $\eps$ even smaller, we can furthermore guarantees that 
$W_1$ grows in some direction, despite the presence of the term $- K \eps H^{1+\tilde \zeta}$ in \eref{e:defW2}. Finally,
the condition \eref{e:condW} is guaranteed to hold because we choose $\tilde \zeta < \zeta$.
\end{proof}

As a corollary of Theorem~\ref{theo:tail}, we obtain:

\begin{corollary}
If $\alpha^2\hat C > T_\infty > {3\over 7} \alpha^2 \hat C$, then even though the system admits a unique invariant measure 
$\mu_\star$, the average kinetic energy of the second oscillator is
infinite, that is $\int p_1^2\,\mu_\star(dx) = \infty$.
\end{corollary}

\begin{proof}
Since in this case the expectation of $H$ is infinite and the expectation of $??$ is finite
\end{proof}

\subsection{Integrability and non-integrability in the case $\pmb{k<2}$.}

In this case, we show that the exponential of a suitable fractional power of $H$ is integrable with respect to the invariant
measure. Our positive result is given by:

\begin{theorem}\label{theo:boundfracexp}
For every $k \in (1,2)$ there exists $\delta > 0$ such that
\begin{equ}[e:integralH]
\int_{\R^4} \exp\bigl(\delta H^{{2\over k}-1}(x)\bigr)\,\mu_\star(dx) < \infty\;,
\end{equ}
where $\mu_\star$ is the unique invariant measure for \eref{e:model}.
\end{theorem}

\begin{proof}
Define $W = \exp(\delta V^\kappa)$ for a (small) constant $\delta>0$ and an exponent $\kappa \in (0,1]$ to be determined later (the optimal exponent will turn out to be $\kappa = {2\over k}-1$). Here, $V$ is the function that was previously defined in \eref{e:deffinalV}.
Since $V$ and $H$ are equivalent in the sense that there exist positive constants $C_1$ and $C_2$ such that
\begin{equ}
C_1^{-1} V - C_2 \le H \le C_1 V + C_2\;,
\end{equ}
showing the integrability of $W$ implies \eref{e:integralH} for a possibly different constant $\delta$.

Applying the chain rule \eref{e:chainrule}, we obtain outside of a sufficiently large compact set the bound
\begin{equs}
\CL W &= \delta \kappa  W \bigl(V^{\kappa-1}\CL V + (\delta \kappa V^{2\kappa -2} + (\kappa-1)V^{\kappa-2})\Gamma(V,V)\bigr) \\
&\le \delta \kappa  W V^{\kappa-1} \bigl(\CL V +  2\delta \kappa V^{\kappa -1} \Gamma(V,V)\bigr) \label{e:LWfrac} \;.
\end{equs}
Note now that it follows immediately from \eref{e:deffinalV} and Proposition~\ref{prop:boundsLpsi}
that, outside of some compact set, one has the bounds
\begin{equs}
|\d_{p_0}V| &\le C \bigl(E_0^{1\over 2} + E_0^{1\over 2}E_1^{{3\over 2k} - 1 - {\alpha\over 2}}+ E_1^{{3\over 2k}-1}\bigr)
\le C \bigl(E_0^{1\over 2} + E_1^{1\over 2}\bigr) \le C \sqrt V\;, \\
|\d_{p_1}V| &\le C \bigl(E_1^{1\over 2} + E_0^{1\over 2k} + E_1^{{5\over 2k}-2} + E_0^{1\over 2} E_1^{{3\over 2k}-{3\over 2}}\bigr)
\le C \bigl(E_0^{1\over 2} + E_1^{1\over 2}\bigr) \le C \sqrt V\;,
\end{equs}
so that $\Gamma(V,V) \le CV$. Combining this with \eref{e:finalboundV}, we obtain the existence of constants $c$ and $C$ (possibly depending on $\kappa$, but not depending on $\delta$) such that
\begin{equ}[e:boundLW]
\CL W \le \delta W V^{\kappa-1} \bigl(C + C\delta V^\kappa - c V^{{2\over k}-1}\bigr)\;,
\end{equ}
thus concluding the proof.
\end{proof}

We have the following partial converse to Theorem~\ref{theo:boundfracexp}:

\begin{theorem}
Let $k \in ({4\over 3}, 2)$. Then, there exists $\Delta > 0$ such that
\begin{equ}[e:integralHinfty]
\int_{\R^4} \exp\bigl(\Delta H^{{2\over k}-1}(x)\bigr)\,\mu_\star(dx) = \infty\;,
\end{equ}
where $\mu_\star$ is the unique invariant measure for \eref{e:model}.
\end{theorem}

\begin{proof}
We are again going to make use of Wonham's criterion.
Let $\tilde K$ be a (sufficiently large) constant, define $\kappa = {2\over k}-1 \in (0,{1\over 2})$, set $F(x) = \exp(\Delta H^\kappa(x))$, and set $W_2(x) = \exp\bigl(\tilde K H^\kappa(x)\bigr)$. We then have
the bound
\begin{equs}
{\CL W_2 \over \gamma\kappa \tilde K W_2} &=  H^{\kappa -1} \bigl(T + T_\infty - p_0^2\bigr) + H^{\kappa-2}(\kappa-1+\kappa \tilde K H^\kappa) \bigl(T p_0^2 + T_\infty p_1^2\bigr) \\
&\le C \bigl(1 + H^{2\kappa -1}\bigr)\;,
\end{equs}
for some constant $C>0$. In particular, we have $\CL V \le F$ outside of some compact set, provided that we choose $\Delta > \tilde K$.
Let now $K$ be any constant smaller than $\tilde K$ and set
\begin{equ}
W_1 = \exp \bigl(K\bigl(H^{\kappa} - 2H^{\kappa-1}\HO  + H^{2\kappa-2}\Phi\bigr)\bigr) \eqdef \exp(K \HH)\;,
\end{equ}
where $\HO$ is the function from Proposition~\ref{prop:propertiesH0}. Note that the properties of $\HO$ imply that, outside
of some compact set, one has the bounds
\begin{equ}
1 \le \HO \le (1+\eps)H\;.
\end{equ}
It is clear that $W_2$ is much larger than $W_1$ at infinity,
so that it remains to show that $\CL W_1 > 0$ outside of a compact set for $K$ sufficiently large. We are actually going to show that
there exists a constant $C$ such that $(\CL W_1)/W_1 \ge C H^{2\kappa-1}$ outside of some compact set. Therefore, we
call a function $f$ negligible if, for every $\eps > 0$, there exists a compact set such that $|f| \le \eps H^{2\kappa-1}$ outside
of this set. Note that since we consider the range of parameters such that $\kappa < {1\over 2}$, bounded functions are
\textit{not} negligible in general.

Using the chain rule \eref{e:chainrule}, we have the identity
\begin{equs}
{\CL W_1 \over K W_1} &= \CL \HH + \gamma K \bigl(T(\d_{p_0}\HH)^2 + T_\infty (\d_{p_1}\HH)^2\bigr) \\
&\ge \CL \HH + \gamma K T_\infty (\d_{p_1}\HH)^2\;.
\end{equs}
We first turn to the estimate of $\CL \HH$. Using again \eref{e:chainrule}, we have the identity
\begin{equs}
\CL \HH &= \bigl(\kappa H^{\kappa -1} + 2(1-\kappa)\bigl( H^{\kappa -2}\HO -H^{2\kappa-3}\Phi \bigr)\bigr)\CL H - 2H^{\kappa-1}\CL \HO + H^{2\kappa-2}\CL \Phi \\
&\quad + \gamma (2\kappa-2)(2\kappa-3)H^{2\kappa-4} \bigl(Tp_0^2 + T_\infty p_1^2\bigr) \Phi + \gamma T_\infty (2\kappa-2) H^{2\kappa-3}p_1\d_P\Phi \\
&\quad + \gamma (\kappa-1) H^{\kappa-3} \bigl(\kappa H + 2(2-\kappa)\HO\bigr)\bigl(Tp_0^2 + T_\infty p_1^2\bigr) \\
& \quad + 2\gamma (1-\kappa) H^{\kappa-2} \bigl(Tp_0 \d_{p_0}\HO + T_\infty p_1 \d_{p_1}\HO\bigr)
\end{equs}
We see immediately that all terms except for the ones in the first line are negligible. On that line, the first term involving $\Phi$
is negligible as well since $\Phi$ scales like a power of the energy strictly smaller than one.
Furthermore, it follows from \eref{e:upperH0} that
\begin{equ}
\kappa H^{\kappa -1} + 2(1-\kappa) H^{\kappa -2}\HO \le \bigl(2 - {\kappa \over 2}\bigr) H^{\kappa-1}\;,
\end{equ}
say. Combining this with Proposition~\ref{prop:propertiesH0} and the fact that the inequality $\kappa > (2k-1)({3\over k}-2)$
holds in the range of parameters under consideration, we obtain the lower bound
\begin{equ}
\CL \HH \gtrsim H^{\kappa-1}\Bigl( \gamma \bigl({\kappa\over 2}-2\bigr)p_0^2 + 2(\gamma-2\theta) \hat p_0^2 + 2\theta \HO\Bigr) + H^{2\kappa-2}\CL \Phi\;.
\end{equ}
Using the definition of $\hat p_0$, and choosing $\theta < \gamma \kappa / 8$, we obtain the existence of a constant $C$ such that
\begin{equs}
\CL \HH &\gtrsim H^{\kappa-1}\bigl(2\theta \HO - C \Hf^\kappa(p_1,q_1) \bigr) + H^{2\kappa-2}\CL_0 \Phi\;.
\end{equs}
Here, we also made use of the scaling properties of $\Phi$ in order to replace $\CL\Phi$ by $\CL_0\Phi$.
Note that the constant $C$ appearing in the expression above 
can be made independent of $\theta$ provided that we restrict ourselves to $\theta \le \gamma \kappa / 16$, say.
At this point, we make the choice $M = 2C$ in the definition of $\Phi$, so that we have the lower bound
\begin{equs}
\CL\HH &\gtrsim H^{\kappa-1}\Bigl(2\theta \HO - {M\over 2} \Hf^\kappa(p_1,q_1) \Bigr) + MH^{2\kappa-2}(\Hf(p_1,q_1) - p_1^2) \\
&\gtrsim - {M\over 2} H^{\kappa-1}\Hf^\kappa(p_1,q_1) + MH^{2\kappa-2}(\HO + \Hf(p_1,q_1) - p_1^2) \;,
\end{equs}
where we made use of the fact that, since $\kappa < 1$, for every constant $C$, there is a compact set such that  $H^{\kappa -1}\HO \ge CH^{2\kappa -2} \HO$ outside of that compact set. From the definitions of $H$ and $\HO$, we see that there exists a constant $C$ and
a compact set outside
of which $C\HO + \Hf(p_1,q_1) > {3\over 4} H$, say, so that we finally obtain the lower bound
\begin{equ}[e:lowerLHH]
\CL\HH \gtrsim {M\over 4} H^{2\kappa-1} - MH^{2\kappa-2} p_1^2 \;.
\end{equ}
Let us now turn to the term $(\d_{p_1}\HH)^2$. We have the identity
\begin{equs}
\d_{p_1}\HH &= H^{\kappa-2}\bigl(\kappa H +2(1-\kappa)\HO\bigr) p_1 + 2(\kappa-1)H^{2\kappa-3}\Phi p_1 \\
& \quad - 2H^{\kappa-1} \d_{p_1} \HO + H^{2\kappa-2}\d_P \Phi\;. 
\end{equs}
Using the inequality $(a+b)^2 \ge {a^2\over 2} - b^2$, as well as the bound \eref{e:bounddp1H0}, it follows that 
there exists a constant $C$ such that the bound
\begin{equs}
(\d_{p_1}\HH)^2 \ge {\kappa^2\over 2} H^{2\kappa-2} p_1^2 - 16 \theta^4 H^{2\kappa-2} \HO 
- C H^{2\kappa - 2}\;,
\end{equs}
holds. Combining this bound with \eref{e:lowerLHH}, we obtain the lower bound
\begin{equ}
{\CL W_1\over KW_1} \gtrsim {M\over 4} H^{2\kappa-1} + \Bigl({\gamma KT_\infty \kappa^2 \over 2}-M\Bigr)H^{2\kappa-2} p_1^2
- 16 \gamma K T_\infty \theta^4 H^{2\kappa-2} \HO\;,
\end{equ}
We now choose $K = 2M/(\gamma T_\infty \kappa^2)$ so that the second term is always positive.
The prefactor of the last term is then given by
$32M \theta^4 / \kappa^2$. Choosing $\theta$ small enough so that $\theta^4 < \kappa^2 / 256$, say, we finally obtain the
lower bound
\begin{equ}
\CL W_1 \ge {MK \over 16}H^{2\kappa-1} W_1 > 0\;,
\end{equ}
valid outside of some sufficiently large compact set, as required.
\end{proof}

\section{Convergence speed towards the invariant measure}
\label{sec:Convergence}

In this section, we are concerned with the convergence rates towards the invariant measure in the case
$1 < k \le 2$ where it exists. Our main result will be that $k = {4\over 3}$ is the threshold separating 
between exponential convergence and stretched exponential convergence.

\subsection{Upper bounds}

Our main tool for upper bounds will be the integrability bounds obtained in the previous section, together
with the results recently obtained in \cite{DFG06SG,BakCatGui08:727}.

The results obtained in Section~\ref{sec:Integral} suggest that it is natural to work in spaces of functions
weighted by $\exp(\delta V^\eps)$, where 
$V$ was defined in \eref{e:defV2}. For
$\eps>0$ and $\delta > 0$, we therefore
define the space $\CB(\eps,\delta)$ as the closure of the space of all smooth compactly supported
functions under the norm
\begin{equ}
\|\phi\|_{(\eps,\delta)} = \sup_{x \in \R^4} |\phi(x)| \exp\bigl(-\delta H^\eps(x)\bigr)\;,
\end{equ}
where we used the letter $x$ to denote the coordinates $(p_0,q_0,p_1,q_1)$. Note that the dual
norm on measures is a weighted total variation norm with weight $\exp(\delta H^\eps(x))$. We also say that 
a Markov semigroup $\CP_t$ with invariant measure $\mu_\star$
has a \textit{spectral gap} in a Banach space $\CB$ containing constants if there exist constants $C$ and $\hat \gamma$ such that
\begin{equ}
\|\CP_t \phi - \mu_\star(\phi)\|_\CB \le C e^{-\hat \gamma t} \|\phi\|_\CB\;,\quad \forall \phi \in \CB\;.
\end{equ}

As a consequence of the bounds of Section~\ref{sec:Integral} , we obtain:

\begin{theorem}\label{theo:upperbound1}
Let $k \in (1,2]$ and set $\kappa = {2\over k}-1$.
Then, the semigroup $\CP_t$ extends to a $\CC_0$-semigroup on the space $\CB(\eps,\delta)$, provided that 
$\eps \le \max\bigl\{{1\over 2},1-\kappa\bigr\}$. Furthermore:
\begin{claim}
\item[a.] If $1 < k < {4\over 3}$ then, for every $\eps \in [1-\kappa, \kappa)$ and every $\delta > 0$,
the semigroup $\CP_t$ has a spectral gap in $\CB(\eps,\delta)$. Furthermore, there exists $\delta_0 > 0$ such that
it has a spectral gap in $\CB(\kappa,\delta)$ for every $\delta \le \delta_0$. 

In particular, for every $\delta>0$
there exist constants $C>0$ and $\hat \gamma > 0$ such that the bound
\begin{equ}[e:conv]
\|\CP_t(x,\cdot\,) - \mu_\star\|_\TV \le C\exp(\delta H^{1-\kappa}(x))e^{-\hat \gamma t}\;,
\end{equ}
holds uniformly over all initial conditions $x$ and all times $t\ge 0$.
\item[b.] If $k={4\over 3}$ then there exists $\delta_0 > 0$ such that
the semigroup $\CP_t$ has a spectral gap in $\CB({1\over 2},\delta)$ for every $\delta \le \delta_0$. 
In particular, there exists $\delta>0$ such that the convergence result \eref{e:conv} holds.
\item[c.] For ${4\over 3} < k < 2$, there exist
positive constants $\delta$, $C$ and $\hat \gamma$ such that the bound
\begin{equ}[e:convslow]
\|\CP_t(x,\cdot\,) - \mu_\star\|_\TV \le C\exp(\delta H^{\kappa}(x))e^{-\hat \gamma t^{\kappa/(1-\kappa)}}\;,
\end{equ}
holds uniformly over all initial conditions $x$ and all times $t\ge 0$.
\item[d.] For the case $k=2$, set $\zeta_\star$ as in Theorem~\ref{theo:tail}. Then, for every $T_\infty < \alpha^2 \hat C$ and every $\zeta < \zeta_\star$, there exists  $C>0$ such that the bound
\begin{equ}[e:convslow2]
\|\CP_t(x,\cdot\,) - \mu_\star\|_\TV \le C H^{1+\zeta}(x) t^{-\zeta}\;,
\end{equ}
holds uniformly over all initial conditions $x$ and all times $t\ge 0$.
\end{claim}
\end{theorem}

\begin{proof}
The set of bounded continuous functions is dense in $\CB(\eps,\delta)$ and is mapped into itself
by $\CP_t$. Therefore, in order to show that it extends to a $\CC_0$-semigroup on $\CB(\eps,\delta)$,
it remains to verify that:
\begin{claim}
\item[1.] There exists a constant $C$ such that $\|\CP_t\phi\|_{(\eps,\delta)} \le C \|\phi\|_{(\eps,\delta)}$ for every
$t \in [0,1]$ and every bounded continuous function $\phi$.
\item[2.] For every $\phi \in \CC_0^\infty$, one has $\lim_{t \to 0}\|\CP_t\phi - \phi\|_{(\eps,\delta)} = 0$.
\end{claim}
Using the \textit{a priori} bounds on the solutions given by the bound $\CL H \le \gamma(T+T_\infty)$,
it is possible to check that the second statement holds for every $(\eps,\delta)$. The first claim then follows
from \cite{MeyTwe93} and \eref{e:boundLW}.

It remains to show claims a to d. Claims a and b follow immediately from
\eref{e:boundLW}. To show that claim c
also holds, we use the fact that, by using \eref{e:boundLW} in the case $\eps = {2\over k}-1$, 
we can find $\delta > 0$ such that the bound
\begin{equ}
\CL W \le - \delta^2 V^{2\kappa-1} W = -\delta^{k \over 2-k} W (\log(W))^{2- {1\over \kappa}}\;,
\end{equ}
holds outside of some compact subset of $\R^4$.
Since we are considering a regular Markov process, every compact set is petite. This shows that
there exists a constant $\delta$ such that, in the terminology of \cite{BakCatGui08:727}, $W$ is a $\phi$-Lyapunov
function for our model with
\begin{equ}
\phi(t) = \delta^{k\over 2-k} t (\log t)^{2-{1\over \kappa}}\;. 
\end{equ}
In particular, this yields the identity
\begin{equ}
H_\phi(t) = \int_1^t {ds \over \phi(s)} = \delta^{-{k \over 2-k}}\int_0^{\mathrm{log}(t)} s^{{1\over \kappa}-2}ds
= C (\log t)^{1-\kappa \over \kappa}\;,
\end{equ}
for some constant $C$ depending on $\delta$ and $\kappa$. It follows from the results in \cite{BakCatGui08:727} 
that the convergence rate to the invariant measure is given by 
\begin{equ}
\psi(t) = {1\over (\phi \circ H_\phi^{-1})(t)} = Ct^{1-2\kappa \over 1-\kappa}e^{-\gamma t^{\kappa/(1-\kappa)}}\;,
\end{equ}
for some positive constants $C$ and $\gamma$, so that \eref{e:convslow} follows.

The case $k = 2$ can be treated in a very similar way. It follows from the first part of the proof of Theorem~\ref{theo:tail}
that, there exists $\beta >0$ and a function $W$ growing like $H^{1+\zeta}$ at infinity 
such that one has the bound $\CL W \le -\beta H^\zeta$ outside of some sufficiently large compact set.
Therefore, $W$ is a $\phi$-Lyapunov function for $\phi(t) = - {\beta} t^{\zeta \over 1+\zeta}$.
Following the same calculations as before, we obtain $\psi(t) = C t^{-\zeta}$, so that the required bound follows at once.
\end{proof}

\subsection{Lower bounds}

In order to be able to use Theorem~\ref{theo:lowerBound}, we
 need upper bounds on the moments of some observable that is not integrable
with respect to the invariant measure. This is achieved by the following proposition:

\begin{proposition}\label{prop:boundsHt}
For every $\alpha >0$ and every $\kappa \in [0,\hf]$ there exist constants $C_\alpha$ and $C_\kappa$ such that
the bounds
\begin{equs}
\bigl(\CP_t H^\alpha\bigr)(x) &\le (H(x) + C_\alpha t)^\alpha\;,\\
\bigl(\CP_t \exp \alpha H^\kappa \bigr)(x) &\le \exp \bigl(\alpha H^\kappa(x) + C_\kappa(1+t)^{\kappa/(1-\kappa)}\bigr)\;,
\end{equs}
hold for every $t>0$ and every $x \in \R^4$.
\end{proposition}

\begin{proof}
Note first that $\CL H \le \gamma(T+T_\infty)$ and that 
\begin{equ}[e:boundL2H]
T (\d_{p_0}H)^2 + T_\infty (\d_{p_1}H)^2 = T p_0^2 + T_\infty p_1^2 \le 2(T+T_\infty) H\;.
\end{equ}
It follows that
for $\alpha \ge 1$, there exists $C>0$ such that one has the bound
\begin{equs}
{d\over dt} \bigl(\CP_t H^\alpha\bigr)(x) &= \bigl(\CP_t (\CL H^\alpha)\bigr)(x) \\
&= \alpha \bigl(\CP_t (H^{\alpha-1}\CL H + \gamma (\alpha-1)H^{\alpha-2}(T p_0^2 + T_\infty p_1^2))\bigr)(x) \\
&\le C \bigl(\CP_t H^{\alpha-1}\bigr)(x) \le C \bigl(\bigl(\CP_t H^{\alpha}\bigr)(x)\bigr)^{1-{1\over \alpha}}\;.
\end{equs}
The last inequality followed from the concavity of $x \mapsto x^{1-{1\over \alpha}}$. Setting $C_\alpha = C/\alpha$, the
bound on $\CP_t H^\alpha$ now follows from a simple differential inequality. The corresponding bound for $\alpha \in (0,1)$
follows by a simple application of Jensen's inequality. 

The bounds on the exponential of the energy are obtained in a similar way. Set $f_\kappa(x) = x (\log x)^{2 -{1 \over \kappa}}$
and note that there exists a constant $K_\kappa$ such that, provided that $\kappa \in (0,\hf]$, $f_\kappa$ is concave
for $x \ge \exp(\alpha K_\kappa^\kappa)$. It then follows as before from \eref{e:boundL2H} and the bound on $\CL H$ that there exists
a constant $C$ such that
\begin{equs}
{d\over dt} \bigl(\CP_t \exp \alpha (K_\kappa + H)^\kappa\bigr)(x) & \le C \bigl(\CP_t (K_\kappa + H)^{2\kappa-1} \exp \alpha (K_\kappa +H)^\kappa\bigr)(x)\\
&= C\bigl(\CP_t f_\kappa(\exp \alpha (K_\kappa + H)^\kappa)\bigr)(x) \\
& \le C f_\kappa \bigl( \bigl(\CP_t \exp \alpha (K_\kappa + H)^\kappa\bigr)(x)\bigr)\;. \label{e:boundexpH}
\end{equs}
The result then follows again from a simple differential inequality.
\end{proof}

As a consequence, we have the following result in the case $k=2$:
\begin{theorem}
For every $\zeta > \zeta_\star$ and every $x_0 \in \R^4$, there exists a constant $C$ and a sequence $t_n$ increasing
to infinity such that $\|\mu_\star - \mu_{t_n}\| \ge C t_n^{-\zeta}$. 
\end{theorem}

\begin{proof}
Let $\tilde \zeta \in (\zeta_\star, \zeta)$, and let $\eps > 0$, $\alpha > \zeta(1+\eps)$. It then follows from Theorem~\ref{theo:tail} and Proposition~\ref{prop:boundsHt} that the assumptions of Theorem~\ref{theo:lowerBound}  are satisfied
with $W(x) = H^{\tilde \zeta}(x)$, $h(s) = s^{-1 -\eps}$, $F(s) = s^{\alpha/\tilde\zeta}$, and
$g(x_0,t) = (H(x_0) + Ct)^{\alpha}$. Applying Theorem~\ref{theo:lowerBound} yields the lower bound
\begin{equ}
\|\mu_\star - \mu_{t_n}\| \ge C t_n^{-{(1+\eps)\alpha \tilde\zeta\over \alpha - \tilde\zeta - \eps\tilde\zeta}}\;,
\end{equ}
for some $C>0$ and some sequence $t_n$ increasing to infinity. Choosing $\eps$ sufficiently small and 
$\alpha$ sufficiently large, we can ensure that the exponent appearing in this expression is larger than $-\zeta$, so that the
claim follows.
\end{proof}

Furthermore, we have
\begin{theorem}\label{theo:lowerboundstretched}
Let $k \in ({4\over 3}, 2)$ and define $\kappa = {2\over k}-1$. Then, there exists a constant $c$ such that,
for every initial condition $x_0 \in \R^4$ there exists a constant $C$ and a sequence of times
$t_n$ increasing to infinite such that $\|\mu_\star - \mu_{t_n}\| \ge C \exp(-c t_n^{\kappa/(1-\kappa)})$. 
\end{theorem}

\begin{proof}
We apply Theorem~\ref{theo:lowerBound} in a similar way to above, but it turns out that we don't need to make such `sharp'
choices for $h$ and $F$. Take $h(s) = s^{-2}$, $F(s) = s^3$, and let $W = \exp(K H^\kappa)$ with the constant
$K$ large enough so that $W$ is not integrable with respect to $\mu_\star$. It then follows from Proposition~\ref{prop:boundsHt}
that we can choose $g(x, t) = \exp \bigl(3K H^\kappa(x) + C (1+t)^{\kappa/(1-\kappa)}\bigr)$ for a suitable constant $C$. 
The  requested bound follows at once, noting that $h \circ (F\cdot h) \circ g = 1/g^{2}$.
\end{proof}

\section{The case of a weak pinning potential}
\label{sec:smallk}

In this section, we are going to study the case $k \le 1$, that is when we have either $V_1 \approx V_2$ or $V_1 \ll V_2$ at infinity.
This case was studied extensively in the previous works \cite{EckHai00NES,ReyTho02ECT,EckHai03SPH,Car07:1076}, but the results
and techniques obtained there do not seem to cover the situation at hand where one of the heat baths is at `infinite temperature'.
Furthermore, these works do not cover the case $k < 1/2$ where one does not have a spectral gap and exponential convergence fails.
One further interest of the present work is that, unlike in the above-mentioned works, we are able to work with the generator $\CL$ instead
of having to obtain bounds on the semigroup $\CP_t$. This makes the argument somewhat cleaner.

We divide this part into two subsections. We first treat the case where one can find a spectral gap, which is relatively easy in the present
setting. 
In the second part, we then treat the case where the spectral gap fails to hold, which follows more closely the heuristics set out in Section~\ref{sec:heuristic2}. 
There, we also show that, rather unsurprisingly, no invariant measure exists in the case where $k \le 0$.

\subsection{The case $\pmb{k > 1/2}$}

Our aim is to find a modified version $\hat H$ of the energy function $H$ such that, for a sufficiently small constant $\beta_0$,
one has $\exp(-\beta_0 \hat H) \CL \exp(\beta_0 \hat H) \ll 0$ at infinity. This is achieved by the following result:
\begin{theorem}\label{thm:k<1}
Let $k \in ({1\over 2},1)$ and let $\delta \in [{1\over k}-1, 1]$. Then, there exist constants $c, C>0$, $\beta_0>0$ and a function $\hat H\colon \R^4 \to \R$
such that 
\begin{claim}
\item The bounds $c H \le \hat H \le CH$ hold outside of some compact set.
\item For any $t>0$, the operator $\CP_t$ admits a spectral gap in the space of measurable functions weighted by $\exp(\beta_0 \hat H^\delta)$.
\end{claim}
\end{theorem}

\begin{remark}
Combining this result with Proposition~\ref{prop:nonintExpH} shows the existence of constants $c, C>0$ such that 
$\int \exp(cH)\,d\mu_\star < \infty$, but $\int \exp(CH)\,d\mu_\star = \infty$.
\end{remark}

\begin{remark}
The technique used in the proof of Theorem~\ref{thm:k<1} is more robust than that used in the previous sections. In particular,
it applies to chains of arbitrary length. It would also not be too difficult to modify it to suit the more general
class of potentials considered in \cite{ReyTho02ECT,Car07:1076}.
\end{remark}

\begin{proof}
Define the variable $y = (q, p_0, p_1)$ with $q = (q_0-q_1)/2$ and let $A$ and $B$ be the matrices defined by
\begin{equ}[e:defAB]
A \eqdef 
\begin{pmatrix}
	0 & {1\over 2} & -{1\over 2} \\ -2\alpha & -\gamma & 0 \\ 2\alpha & 0 & 0
\end{pmatrix}\;,\qquad
B \eqdef 
\sqrt{2\gamma} \begin{pmatrix}
	0 & 0 \\ \sqrt{T} & 0 \\ 0 & \sqrt{T_\infty}
\end{pmatrix}\;.
\end{equ}
With this notation, we can write the equations of motion for $y$ following from \eref{e:model} as
\begin{equ}[e:dyny]
dy = Ay\,dt + F(y, Q)\, dt + B\, dw(t)\;,
\end{equ}
where we defined the centre of mass $Q = (q_0+q_1)/2$ and $F \colon \R^4 \to \R^3$ is a vector-valued function whose components are all bounded by $C + |V_1'(q_0)|+ |V_1'(q_0)|$
for some constant $C$.

Since $\det A = -\gamma\alpha < 0$ and we know from a simple contradiction argument \cite{ReyTho02ECT,Car07:1076}
that the energy of the system converges to zero under the deterministic equation $\dot y = Ay$, we conclude that all eigenvalues of $A$ have strictly
negative real part. As a consequence, there exists $\tilde \gamma > 0$ such that the strictly positive definite symmetric quadratic form
\begin{equ}[e:defS]
\scal{y,Sy} \eqdef \int_0^\infty e^{\tilde \gamma t}\|e^{At}y\|^2\,dt
\end{equ}
is well-defined. A simple change of variable shows that one then has the bound 
\begin{equ}[e:boundexpA]
\scal{e^{At},Se^{At}y} \le e^{-\tilde \gamma t}\scal{y,Sy}\;.
\end{equ}
For any given (small) value $\eps >0$, let now $G_\eps\colon \R \to \R$ be a smooth function such that:
\begin{claim}
\item There exists a constant $C_\eps$ such that the bounds $G_\eps(q)V_1'(q) \le C_\eps - |V_1'(q)|^2$ and $|G_\eps(q)|^2 \le C_\eps + |V_1'(q)|^2$ 
hold for every $q\in \R$.
\item One has $|G_\eps'(q)| \le \eps$ for every $q \in \R$.
\end{claim}
Since we assumed that $k < 1$, it is possible to construct a function $G_\eps$ satisfying these conditions by choosing $R_\eps$ sufficiently
large, setting $G_\eps(q) = - V_1'(q)$ for
$|q| \ge 2R_\eps$, $G_\eps(q) = q |R_\eps|^{2k-2}$ for $|q| \le R_\eps$, and interpolating smoothly in between. For large values of $R_\eps$,
one can then guarantee that $|G_\eps'(q)| \le C R_\eps^{2k-2}$, which does indeed go to $0$ for large values of $R_\eps$.

We now define, for a (large) constant $\xi$ to be determined,
\begin{equ}
\hat H = H + \scal{y,Sy} - \xi (p_0+p_1)\bigl(G_\eps(q_0) + G_\eps(q_1)\bigr)\;.
\end{equ}
Before we bound $\CL \hat H$, we note that we have the bound
\begin{equs}
\bigl(G_\eps(q_0) + G_\eps(q_1)\bigr)&\bigl(V_1'(q_0) + V_1'(q_1)\bigr)  = 2 \bigl(G_\eps(q_0) V_1'(q_0) + G_\eps(q_1)V_1'(q_1)\bigr) \\
&\qquad + \Bigl(\int_{q_0}^{q_1} G_\eps'(q)\,dq\Bigr) \bigl(V_1'(q_0) - V_1'(q_1)\bigr)\\
&\le 2C_\eps - 2 \bigl(|V_1'(q_0)|^2 + |V_1'(q_1)|^2\bigr) + C\eps (q_0-q_1)^2\;,
\end{equs}
for some constant $C$ independent of $\eps$.

It therefore follows from \eref{e:boundexpA}, \eref{e:dyny}, \eref{e:model} and the properties of $G_\eps$ 
that there exist constants $C_i$ independent of $\xi$ and $\eps$
such that we have the bound
\begin{equs}
\CL \hat H &\le C_1 - \gamma p_0^2 - \tilde\gamma \scal{y,Sy} + 2\scal{y, S F(y,Q)} \\
&\quad + \xi \bigl(G_\eps(q_0) + G_\eps(q_1)\bigr)\bigl(V_1'(q_0) + V_1'(q_1)\bigr) \\
&\quad +\gamma \xi p_0 \bigl(G_\eps(q_0) + G_\eps(q_1)\bigr) - \xi (p_0+p_1)\bigl(G_\eps'(q_0)p_0 + G_\eps'(q_1)p_1\bigr) \\
&\le C_2 \bigl(C_\eps+|V_1'(q_0)|^2 + |V_1'(q_1)|^2\bigr) \\
&\quad - {\tilde \gamma  - C_3\eps\xi\over 2}\scal{y,Sy} - 2\xi \bigl(|V_1'(q_0)|^2 + |V_1'(q_1)|^2\bigr) \;.
\end{equs}
It follows that, by first making $\xi$ sufficiently large and then making  $\eps$ sufficiently small, it is possible to obtain the bound
\begin{equ}[e:LhatH]
\CL \hat H \le C - {\tilde\gamma \over 2} \bigl(1+\scal{y,Sy} + |V_1'(q_0)|^2 + |V_1'(q_1)|^2\bigr)\;,
\end{equ}
for some constant $C$. (The constant $C$ depends of course on the choice of $\xi$ and of $\eps$, but assume those to be fixed from now on.) 
Furthermore, it follows immediately from the definition of $\hat H$ that 
\begin{equ}[e:GammahatH]
\Gamma(\hat H, \hat H) \le C \bigl(1+\scal{y,Sy} + |V_1'(q_0)|^2 + |V_1'(q_1)|^2\bigr) \le C \hat H^{2-{1\over k}}\;,
\end{equ}
where we used the scaling behaviour of $V_1$ in order to obtain the second bound. Set now $W = \exp(\beta_0 \hat H^\delta)$ for a constant $\beta_0$ to be determined. It follows from \eref{e:LWfrac} that the bound
\begin{equ}
\CL W \le \beta_0 \delta W \hat H^{\delta-1} \bigl(\CL \hat H + 2\beta_0 \delta \hat H^{\delta -1} \Gamma(\hat H,\hat H)\bigr)  \;,
\end{equ}
holds outside of some sufficiently large compact set.
Combining this with \eref{e:GammahatH} and \eref{e:LhatH}, we see that if $\delta \in [{1\over k}-1,1]$ and $\beta_0$ is sufficiently small, 
then the bound
\begin{equ}
\CL W \le - C  W (\hat H)^{\delta + 1 - {1\over k}} \le  - C W\;,
\end{equ}
holds outside of some compact set. The claim then follows immediately from Theorem~\ref{theo:positiveAbstract3}.
\end{proof}

The case $k=1$ can be shown similarly, but the result that we obtain is slightly stronger in the sense that one has a spectral
gap in spaces weighted by $H^\delta$ for any $\delta>0$:

\begin{theorem}\label{thm:k=1}
Let $k =1$ and let $\delta >0$. Then, for any $t>0$, the operator $\CP_t$ admits a spectral gap in the space of measurable functions weighted by $H^\delta$.
\end{theorem}

\begin{proof}
The proof is similar to the above, but this time by setting $\tilde y = (q_0,q_1,p_0,p_1)$,
\begin{equ}
\tilde A \eqdef 
\begin{pmatrix}
	0 & 0 & 1 & 0 \\ 0 & 0 & 0 & 1 \\ -\alpha &\alpha & -\gamma & 0 \\ \alpha &-\alpha & 0 & 0
\end{pmatrix}\;,\qquad
\tilde B \eqdef 
\sqrt{2\gamma} \begin{pmatrix}
	0 & 0 \\ 0 & 0 \\ \sqrt{T} & 0 \\ 0 & \sqrt{T_\infty}
\end{pmatrix}\;,
\end{equ}
 and noting that 
\begin{equ}
d\tilde y = \tilde A\tilde y\,dt + F(\tilde y)\, dt + \tilde B\, dw(t)\;,
\end{equ}
for some bounded function $F$. It then suffices to construct $\tilde S$ similarly to above and to set $\hat H = \scal{\tilde y, \tilde S \tilde y}$, without
requiring any correction term. This yields the  existence of constants $C_1$ and $C_2$ such that one has the bounds
\begin{equ}
\CL \hat H \le - C_1 \hat H\;,\qquad \Gamma(\hat H, \hat H) \le C_2 \hat H\;,
\end{equ}
outside of some compact set. The existence of a spectral gap in spaces weighted by $\hat H^\delta$ follows at once. The claim then 
follows from the fact that $\hat H$ is bounded from above and from below by multiples of $H$. 
\end{proof}

\subsection{The case $\pmb{k \le 1/2}$}

This case is slightly more subtle since the function $V'(q)$ is either bounded or even converges to zero
at infinity, so that bounds of the type \eref{e:LhatH} are not very useful. We nevertheless have the following result: 

\begin{theorem}\label{thm:k<12}
Let $k \in (0,{1\over 2}]$. Then, \eref{e:model} admits a unique invariant probability measure $\mu_\star$ and
 there exist constants $c, C>0$, $\beta_0>0$, and a function $\hat H\colon \R^4 \to \R$
such that 
\begin{claim}
\item The bounds $c H \le \hat H \le CH$ hold outside of some compact set.
\item If $k={1\over 2}$, then $\CP_t$ admits a spectral gap in the space of measurable functions weighted by $\exp(\beta_0 \hat H)$.
\item If $k < {1\over 2}$, then  there exist
positive constants $C$ and $\hat \gamma$ such that the bound
\begin{equ}[e:convslow3]
\|\CP_t(x,\cdot\,) - \mu_\star\|_\TV \le C\exp(\beta_0 H(x))e^{-\hat \gamma t^{k/(1-k)}}\;,
\end{equ}
holds uniformly over all initial conditions $x$ and all times $t\ge 0$.
\end{claim}
\end{theorem}

\begin{proof}
Define again $y$, $A$ and $B$ as in \eref{e:defAB} but let us be slightly more careful about the remainder term.
We define as before the center of mass $Q = (q_0 + q_1)/2$ and the displacement $q = (q_0 - q_1)/2$ and write
\begin{equ}
V_1'(q_0) = V_1'(Q) + R_0(q,Q)\;,\quad V_1'(q_1) = V_1'(Q) + R_1(q,Q)\;.
\end{equ}
With this notation, defining furthermore the vector $\one = (0, 1, 1)$, the equation of motion for $y = (q,p_0,p_1)$ is
given by
\begin{equ}
dy = Ay\,dt - V_1'(Q) \one\,dt + R(Q,y)\,dt + B\,dw(t)\;,\quad R = (0, -R_0(Q,y), -R_1(Q,y))\;.
\end{equ}
This suggests the introduction, for fixed $Q \in \R$, of the reduced generator $\CL_Q$ acting on functions from $\R^3$ to $\R$ by
\begin{equ}
\CL_Q = \scal{Ay , \d_y} - V_1'(Q) \scal{\one, \d_y} + {1\over 2} \scal{B^*\d_y, B^*\d_y}\;.
\end{equ}
Following the usual procedure in the theory of homogenisation, we wish to correct the `slow variable' $Q$ in order to obtain
an effective equation that takes into account the behaviour of the `fast variable' $y$. Since the equation of motion for $Q$ is given by
$\dot Q = (p_0 + p_1)/2 = \scal{\one,y}/2$, this can be achieved by finding a function $\psi(Q)$ such that
$\scal{\one,y}/2 - \psi(Q)$ is centred with respect to the invariant measure for $\CL_Q$ and then
solving the Poisson
equation $\CL_Q \phi_Q = \scal{\one,y}/2 - \psi(Q)$. 

Since all the coefficients of $\CL_Q$ are linear (remember that $Q$ is a constant there),
this can be solved explicitly, yielding
\begin{equ}
\psi(Q) = -{2\over \gamma}V_1'(Q)\;,\qquad \phi_Q(y) = - \scal{a,y}\;,\quad a = (1,1/\gamma,1/\gamma)\;.
\end{equ}
We now introduce the corrected variable $\hat Q = Q + \scal{a,y}$, so that the equations of motion for $\hat Q$ are given by
\begin{equ}
d\hat Q = - {2\over \gamma} V_1'(Q)\, dt + \scal{a, R(Q,y)}\,dt + \sqrt{2T\over \gamma}dw_0(t) + \sqrt{2T_\infty \over \gamma}dw_1(t)\;.
\end{equ}
Defining $\hat \gamma = {2\over \gamma}$, the `mean temperature' $\hat T = (T + T_\infty)/2$, and
\begin{equ}[e:defRhat]
\hat R = \scal{a, R(Q,y)} + \hat \gamma \bigl(V_1'(\hat Q) - V_1'(Q)\bigr)\;,
\end{equ}
we thus see that there exists a Wiener process $W$ such that $\hat Q$ satisfies the equation
\begin{equ}
d\hat Q = -\hat \gamma V_1'(\hat Q)\,dt + \hat R\,dt + \sqrt{2\hat\gamma \hat T} \, dW(t)\;.
\end{equ}
Setting again $S$ as in \eref{e:defS}, this suggests that in order to extract the tail behaviour of the invariant measure
for \eref{e:model}, a good test function would be $V_1(\hat Q) + \scal{y,Sy}$. This function however
turns out not to be suitable in the regime where $\hat Q$ is large and $y$ is small, because of the constant appearing when applying
$\CL$ to $\scal{y,Sy}$. In order to avoid this, let us introduce a smooth increasing 
function $\chi\colon \R_+ \to [0,1]$ such that
$\chi(t) = 1$ for $t \ge 2$ and $\chi(t) = 0$ for $\chi \le 1$.
We also define the function $\scal{\hat Q} = \sqrt{1+\hat Q^2}$, so that $|V_1'(\hat Q)| \le C\scal{\hat Q}^{2k-1}$
and similarly for $V_1''(\hat Q)$.

Note that, since we are considering the regime where $V_1'$ is a bounded function, there exists a constant $C_S$ such that
\begin{equ}
-C_S - 2\tilde \gamma \scal{y,Sy} \le \CL  \scal{y,Sy} \le C_S - {\tilde \gamma \over 2} \scal{y,Sy}\;,
\end{equ}
where $\tilde \gamma$ is as in \eref{e:boundexpA}. Furthermore, we note that since all terms contained in $\hat R$ 
are of the form $V_1'(\hat Q) - V_1'(\hat Q + \scal{b,y})$ for some vector $b \in \R^3$, 
there exists a constant $C$ such that the bound
\begin{equ}[e:boundhatR]
|\hat R| \le \left \{ \begin{array}{cr@{\ }l}
C & |\hat Q| &\le C|y|\;, \\
C|y| \scal{\hat Q}^{2k-2} & |\hat Q| &\ge C|y|\;,  
\end{array}\right.
\end{equ}
holds for every pair $(\hat Q, y)$. (In particular $\hat R$ is bounded.)
We now set
\begin{equ}
W = \exp\bigl(\beta_0 \scal{y,Sy}\bigr) + \exp\bigl(\beta_0 \lambda V_1(\hat Q)\bigr)\;,
\end{equ}
for some positive constants $\beta_0$ and $\lambda$ to be determined.

Since we are only interested in bounds that hold outside of a compact set, we 
use in the remainder of this proof the notation $f \lesssim g$ to signify that there exists a constant $c>0$ such that
the bound $f \le c g$ holds outside of a sufficiently
large compact set.
With this notation, one can check in a straightforward way that there exist constants $\beta_i$ depending on 
$\beta_0$ and $\lambda$ such that
the two-sided bound
\begin{equ}
 \exp(\beta_1 H) \lesssim W  \lesssim \exp(\beta_2 H)\;,
\end{equ}
holds.

It follows then from the chain rule that there exist constants $C_i > 0$ such that one has the upper bound
\begin{equs}
\CL W &\le \beta_0 \Bigl(C_S - ({\tilde \gamma \over 2} - C_1\beta_0) \scal{y,Sy}\Bigr)\exp\bigl(\beta_0 \scal{y,Sy}\bigr)\label{e:upperboundW}\\
&\quad + \beta_0\lambda  \bigl(-(1-C_2\beta_0 \lambda) |V_1'(Q)|^2 + C_3 |V_1'(\hat Q)| |\hat R| + C_4V_1''(\hat Q)\bigr)\exp\bigl(\beta_0 \lambda V_1(\hat Q)\bigr)\;.
\end{equs}
Choosing $\beta_0$ sufficiently small, we obtain the existence of a constant $C$ such that the bound
\begin{equs}
\CL W & \lesssim \bigl(C - \scal{y,Sy}\bigr)\exp\bigl(\beta_0 \scal{y,Sy}\bigr)
+ \scal{\hat Q}^{2k-1}\bigl(C|\hat R| -\scal{\hat Q}^{2k-1} \bigr)\exp\bigl(\beta_0 \lambda V_1(\hat Q)\bigr)\;,
\end{equs}
holds. 

We now consider three separate cases. In the regime $\lambda V_1(\hat Q) \ge \scal{y,Sy} \ge C$, it follows
from \eref{e:boundhatR} and our definition of $\lambda$ that we have the bound
\begin{equ}
\CL W \lesssim \scal{\hat Q}^{2k-1}\bigl(C|\hat R| -\scal{\hat Q}^{2k-1} \bigr)\exp\bigl(\beta_0 \lambda V_1(\hat Q)\bigr)
\lesssim - \scal{\hat Q}^{4k-2} W\;.
\end{equ}
In the regime where $\lambda V_1(\hat Q) \ge \scal{y,Sy}$ but $\scal{y,Sy} \le C$, we similarly have
\begin{equ}
\CL W \lesssim C\exp(\beta_0 C)  - \scal{\hat Q}^{4k-2} \exp\bigl(\beta_0 \lambda V_1(\hat Q)\bigr)
\lesssim - \scal{\hat Q}^{4k-2} W\;.
\end{equ}
Finally, in the regime where $\lambda V_1(\hat Q) \le \scal{y,Sy}$, we have the bound
\begin{equ}
\CL W \lesssim - \scal{y, Sy} \exp\bigl(\beta_0 \scal{y,Sy}\bigr) \lesssim - |y|^2 W\;.
\end{equ}
Combining all of these bounds, we have
\begin{equ}
\CL W \lesssim - (\log W)^{2-{1\over k}} W\;,
\end{equ}
so that the upper bounds on the transition probabilities follow just as in the proof of Theorem~\ref{theo:upperbound1}
with $\kappa$ replaced by $k$.
\end{proof}

Before we obtain lower bounds on the convergence speed, we show the following non-integrability result:
\begin{lemma}\label{lem:nonintV}
In the case $k < {1\over 2}$ there exists $\beta > 0$ such that $\int \exp(\beta V_1(\hat Q))\,d\mu_\star = \infty$.
\end{lemma}
\begin{proof}
We are going to construct functions $W_1$ and $W_2$ satisfying Wonham's criterion.
Let $\hat Q$, $S$ and $y$ be as in the proof of the previous result and set
\begin{equ}
W_2 = \exp(2\beta V_1(\hat Q)) + \exp(\eps \scal{y,Sy})\;,
\end{equ}
for constants $\beta>0$ and $\eps>0$ to be determined. It follows from the boundedness of $V_1'$, $V_1''$ and $\hat R$ that,
whatever the choice of $\beta$, one has
\begin{equ}
\CL W_2 \lesssim \exp(\beta V_1(\hat Q))\;,
\end{equ}
provided that we choose $\eps$ sufficiently small. Setting
\begin{equ}
W_1 = \exp(\beta V_1(\hat Q)) - \exp(\eps \scal{y,Sy})\;,
\end{equ}
we have similarly to \eref{e:upperboundW} the bound 
\begin{equs}
\CL W_1 &\ge \beta \bigl(\bigl(C_2 \beta - 1\bigr) |V_1'(\hat Q)|^2 - C_3 |V_1'(\hat Q)| |\hat R| + C_4 V_1''(\hat Q)\bigr) \exp(\beta V_1(\hat Q)) \\
&\quad + \eps\bigl(( \tilde \gamma/2 - C_1\eps)\scal{y,Sy} - C_S\bigr)\exp(\eps \scal{y,Sy})\;,
\end{equs}
so that an analysis similar to before shows that $\CL W_1 \ge 0$ outside of some compact set, provided that $\eps < \tilde \gamma / (2C_1)$
and $\beta > 1/C_2$, thus concluding the proof.
\end{proof}

\begin{remark}
The proof of Lemma~\ref{lem:nonintV} does not require $k > 0$. It therefore shows that there exists no invariant probability measure
for \eref{e:model} if $k \le 0$. 
\end{remark}

We now use this result in order to obtain the following lower bound on the convergence of
the  transition probabilities towards the invariant measure:

\begin{theorem}
Let $k \in (0, {1\over 2})$. Then, there exists a constant $c$ such that,
for every initial condition $x_0 \in \R^4$ there exists a constant $C$ and a sequence of times
$t_n$ increasing to infinite such that $\|\mu_\star - \mu_{t_n}\| \ge C \exp(-c t_n^{k/(1-k)})$. 
\end{theorem}

\begin{proof}
We use the same notations as above. Let $\beta$ be sufficiently large so that the function $\exp(\beta V_1(\hat Q))$ is not
integrable with respect to the invariant measure. We also fix some small $\eps>0$ and we set
\begin{equ}
W = \exp(\beta V_1(\hat Q)) + \exp(\eps\scal{y,Sy})\;.
\end{equ}
We then obtain in a very similar way to before the upper bound
\begin{equs}
\CL W &\le \beta \bigl(\bigl(C_2 \beta - 1\bigr) |V_1'(\hat Q)|^2 + C_3 |V_1'(\hat Q)| |\hat R| + C_4 V_1''(\hat Q)\bigr) \exp(\beta V_1(\hat Q)) \\
&\quad + \eps\bigl(C_S - (\tilde \gamma/2 - C_1\eps)\scal{y,Sy}\bigr)\exp(\eps \scal{y,Sy})\;.
\end{equs}
It follows again from a similar analysis that there exists a constant $C>0$ such that the bound
\begin{equ}
\CL W \le C \bigl(\log W\bigr)^{2- {1\over k}} W\;,
\end{equ}
holds outside of some compact set. As in the proof of Proposition~\ref{prop:boundsHt}, this implies the existence of a constant $C>0$ such
 that one has the pointwise bound
\begin{equ}
\CP_t W \le W \exp\bigl(C(1+t)^{k/(1-k)}\bigr)\;.
\end{equ}
Combining this with Lemma~\ref{lem:nonintV}, the rest of the proof is identical to that of Theorem~\ref{theo:lowerboundstretched}.
\end{proof}

\bibliographystyle{Martin}
\bibliography{./InfiniteTemperature}

\end{document}